\documentclass[a4paper, 11pt]{article}
\usepackage{tikz}

\usepackage{mathpazo}
\usepackage{amsmath,amssymb,amsthm}
\usepackage[latin1]{inputenc}
\usepackage{version,tabularx,multicol}
\usepackage{graphicx,float,psfrag}
\usepackage{stmaryrd}
 \usepackage{color}
\usepackage[pdftex]{hyperref}
\usepackage[numbers,sort&compress]{natbib}

\hypersetup{
    unicode      = false,     
    pdftoolbar   = true,      
    pdfmenubar   = true,      
    pdffitwindow = true,      
    pdfnewwindow = true,      
    colorlinks   = true,      
    linkcolor    = blue,      
    citecolor    = red,      
    filecolor    = blue,      
    urlcolor     = blue       
}

\newcommand{\reff}[1]{(\ref{#1})}

\theoremstyle{plain}
\newtheorem{theorem}{Theorem}[section]

 \newtheorem{proposition}[theorem]{Proposition}
\newtheorem{lem}[theorem]{Lemma}
 \newtheorem{lemma}[theorem]{Lemma}

\newtheorem{remark}{Remark}[section]

\newcommand{\mR}{{\mathbb R}}

  \makeatletter
  \@addtoreset{equation}{section}
  \makeatother

 \def\beqlb{\begin{eqnarray}}\def\eeqlb{\end{eqnarray}}
 \def\beqnn{\begin{eqnarray*}}\def\eeqnn{\end{eqnarray*}}

 \def\mbb{\mathbb}

 \def\qed{\hfill$\Box$\medskip}

\newcommand{\bcen}{\begin{center}}
\newcommand{\ecen}{\end{center}}
\newcommand{\bgeqn}{\begin{equation}}
\newcommand{\edeqn}{\end{equation}}

\def\az{\alpha}

\def\gz{\gamma}

\def\dz{\delta}
\def\lz{\lambda}

\def\ez{\varepsilon}

\def\E{{\mbb  E}}

\def\mR{{\mbb  R}}

\renewcommand{\P}{\mathbb{P}}
\def\rar{\rightarrow}

\def\l{\left}
\def\r{\right}

\begin{document}

\title{Lower deviation and moderate deviation probabilities for maximum of a branching random walk }

\author{Xinxin Chen and Hui He}

\maketitle

\noindent\textit{Abstract:}
 Given a super-critical branching random walk on $\mR$ started from the origin,
 let $M_n$ be the maximal position  of individuals at the $n$-th generation.
 Under some mild conditions, it is known from  \cite{A13} that  as
 $n\rar\infty$, $M_n-x^*n+\frac{3}{2\theta^*}\log n$ converges in law for some suitable constants $x^*$ and $\theta^*$. In this work,
 we investigate its moderate deviation, in other words,  the convergence rates of
 $$\P\left(M_n\leq x^*n-\frac{3}{2\theta^*}\log n-\ell_n\right),$$
 for any positive sequence $(\ell_n)$ such that $\ell_n=O(n)$ and $\ell_n\uparrow\infty$. As a by-product, we also obtain lower deviation of $M_n$; i.e., the convergence rate of
 \[
 \P(M_n\leq xn),
 \] for $x<x^*$ in B\"{o}ttcher case where the offspring number is at least two.  Finally, we apply our techniques to study the small ball probability of limit of derivative martingale. 

\bigskip


\bigskip


\section{Introduction}\label{Sec01}

\subsection{Branching random walk and its maximum}
We consider a discrete-time branching random walk on the real line, which, as a generalized branching process, has been always a very attractive objet in probability theory in recent years. It is closely related to many other random models, for example, random walk in random environment, random fractals and discrete Gaussian free field; see \cite{HS07}, \cite{Liu98}, \cite{Liu06}, \cite{BDZ16a} and \cite{AHS17} references therein. One can refer to \cite{S12} and \cite{Z16} for the recent developments on branching random walk.

Generally, to construct a branching random walk, we take a random point measure as the reproduction law which describes both the number of children and their displacements. Each individual produces independently its children according to the law of this random point measure. In this way, one develops a branching structure with motions.

In this work, we study a relatively simpler model which is constructed as follows. We take a Galton-Watson tree ${\cal T}$, rooted at $\rho$, with offspring distribution given by $\{p_k; k\geq0\}$. For any $u,v\in\mathcal{T}$, we write $u\preceq v$ if $u$ is an ancestor of $v$ or $u=v$. Moreover, to each node $v\in\mathcal{T}\setminus\{\rho\}$, we attach a real-valued random variable $X_v$ to represent its displacement. So the position of $v$ is defined by
\[
S_v:=\sum_{\rho\prec u\preceq v}X_u.
\]
Let $S_\rho:=0$ for convenience. Suppose that given the tree $\mathcal{T}$, $\{X_v; v\in\mathcal{T}\setminus\{\rho\}\}$ are i.i.d. copies of some random variable $X$ (which is called displacement or step size). Note here that the reproduction law is given by $\sum_{|u|=1}\delta_{X_u}$. Thus, $\{S_u; u\in\mathcal{T}\}$ is our branching random walk with independence between offsprings and motions. This independence will be  necessary for our arguments.

For any $n\in\mathbb{N}$, let $M_n$ be the maximal position at the $n$-th generation, in other words,
\[
M_n:=\sup_{|v|=n}S_v,
\]
where $|v|$ denotes the generation of node $v$, i.e., the graph distance between $v$ and $\rho$. The asymptotics of $M_n$ have been studied by many authors, both in the subcritical/critical case and in supercritical case. One can refer to  \cite{LS15}, \cite{NZ17} and \cite{S12} for more details.

We are interested in the supercritical case where $\sum_{k\geq0}kp_k>1$ and the system survives with positive probability. Let $(S_n)$ be a random walk started from $0$ with i.i.d. increments distributed as $X$. Observe that for any individual $|u|=n$ of the $n$-th generation, $S_u$ is distributed as $S_n$. If $\E[|X|]<\infty$, classical law of large number tells us that $S_n\sim \E[X]n$ almost surely. However, as there are too many individuals in this supercritical system, the asymptotical behavior of $M_n$ is not as that of $S_n$.

Conditionally on survival, under some mild conditions, it is known from \cite{Ha74, Ki75, B76} that
\[
\frac{M_n}{n}\rightarrow x^*>\E[X],\quad \textrm{ a.s.,}
\]
where $x^*$ is a constant depending on both offspring law and displacement. Later, the logarithmic order of $M_n-x^*n$ is given by \cite{AR09}, \cite{HS09} in different ways. A\"id\'ekon in \cite{A13} showed that $M_n-x^*n+\frac{3}{2\theta^*}\log n$ converges in law for some suitable $\theta^*\in\mbb R^*_+$, which is an analogue of Bramson's result for branching Brownian motion in \cite{Br78}; see also \cite{BDZ16b}. More details on these results will be given in Section 2.

 For maximum of branching Brownian motion, Chauvin and Rouault \cite{CR99} first studied the large deviation probability. Recently, Derrida and Shi \cite{DS16, DS17a, DS17b} considered both the large deviation and lower deviation. They established precise estimations. On the other hand, for branching random walk, Hu in \cite{Hu16} studied the moderate deviation for $M_n-x^*n+\frac{3}{2\theta^*}\log n$; i.e.; $\P(M_n\leq x^*n-\frac{3}{2\theta^*}\log n-\ell_n)$ with $\ell_n=o(\log n)$. Later, Gantert and H\"ofelsauer \cite{GH18} and Bhattacharya \cite{Bh18} investigated large deviation probability $\P(M_n\geq xn)$ for $x>x^*$. In the same paper \cite{GH18}, Gantert and H\"ofelsauer also studied the lower deviation probability  $\P(M_n\leq xn)$ for $x<x^*$  mainly in Schr\"oder case when $p_0+p_1>0$. In fact, branching random walk in Schr\"oder case can be viewed as a generalized version of branching Brownian motion. 

Motivated by \cite{Hu16}, \cite{GH18} and \cite{DS17a}, the goal of this article is to study moderate deviation $\P(M_n\leq x^*n-\frac{3}{2\theta^*}\log n-\ell_n)$ with $\ell_n=O(n)$.  As we already mentioned, \cite{Hu16} first considered this problem with $\ell_n=o(\log n)$; see Remarks \ref{Hu01} and \ref{Hu02} below for more details. As a by-product of our main results, in B\"ottcher case when $p_0=p_1=0$, we also obtain the lower deviation of $M_n$, i.e., $\P(M_n\leq xn)$ for $x<x^*$, which completes the work \cite{GH18}.  We shall see that the lower deviation of $M_n$ in B\"ottcher case  turns to be very different from that in Schr\"oder case. In fact, Gantert and H\"ofelsauer \cite{GH18} proved that in Schr\"oder case $\P(M_n\leq xn)$ decays exponentially. On contrast, in  B\"ottcher case, we shall show that $\P(M_n\leq xn)$ may decay double-exponentially or super-exponentially depending on the tail behaviors of step size $X$. We will consider three typical left tail distributions of the step size $X$ and obtain the corresponding decay rates and rate functions.  Finally, we also apply our techniques to study the small ball probability for the limit of derivative martingale. The corresponding problem was also considered in \cite{Hu16} for a class of Mandelbrot's cascades in B\"ottcher case with bounded step size and in Schr\"oder case; see also \cite{Liu99} and \cite{Liu01} for more backgrounds. 
Let us state the theorems in the following subsection.

As usual, $f_n=O(g_n)$ or $f_n=O(1)g_n$ means that $f_n\leq Cg_n$ for all $n\geq1$. $f_n=\Theta(1)g_n$ means that $f_n$ is bounded above and below by  a positive and finite constant multiple of $g_n$  for all $n\geq1$. $f_n=o(g_n)$ or $f_n=o_n(1)g_n$ means $\lim_{n\rar\infty}\frac{f_n}{g_n}=0$.

\subsection{Main results}
Suppose that we are in the supercritical case where the tree $\cal T$ survives with positive probability. Formally, we assume that for the offspring law $\{p_k\}_{k\geq0}$:
\begin{equation}\label{momoffspring}
m:=\sum_{k\geq0}kp_k>1\textrm{ and }  \sum_{k\geq0} k^{1+\xi}p_k<\infty, \textrm{ for some }\xi>0.
\end{equation}
At the same time, suppose that for the step size $X$,
\begin{equation}\label{expmom}
\E[X]=0,\textrm{ and }\psi(t):=\E[e^{tX}]<\infty,\quad \forall t\in(-K, K),
\end{equation}
for some $K\in(0,\infty]$.
 We define the rate function of large deviation for the corresponding random walk $\{S_n\}$ with i.i.d. step sizes $X$ by
\begin{equation*}
I(x):=\sup_{t\in\mbb R}\{tx-\log \psi(t)\},\quad \forall x\in\mbb R.
\end{equation*}
 Then it is known from Theorem 3.1 in \cite{B10} that under \reff{momoffspring} and \reff{expmom},
 \[
 \frac{M_n}{n}\rightarrow x^*, \quad \P-a.s.,
 \]
 where $x^*=\sup\{x\geq0: I(x)\leq \log m\}\in (0,\infty)$. Note that if $x^*<\text{ess sup}X\in (0,\infty]$, then $I(x^*)=\log m$ since $I$ is continuous in $(0, \text{ess sup}X)$, and
\begin{equation}\label{hyp3}
\exists\ \theta^*\in(0,\infty) \textrm{ such that }I(x^*)=\theta^*x^*-\log\psi(\theta^*)=\log m.
 \end{equation}
 According to Theorem 4.1 in \cite{B10}, it further follows from \reff{hyp3} that $\P$-a.s.,
 \begin{equation*}\label{roughcvg}
 M_n-nx\rightarrow-\infty,
 \end{equation*}
which  fails if  $x^*=\text{ess sup}X\in(0,\infty)$ and $m\P(X=x^*)>1$.
  Besides \reff{momoffspring}, \reff{expmom} and \reff{hyp3}, if we further suppose that
  \begin{equation}\label{hyp4}
  \psi(t)<\infty,\quad \forall t\in(-K, \theta^*+\delta)\textrm{ for some }\delta>0.
  \end{equation}
  then  it is shown in \cite{AR09} and \cite{HS09}  that $M_n=m_n+o_\P(\log n)$ where
  \begin{equation}
  m_n:=x^*n-\frac{3}{2\theta^*}\log n,\quad \forall n\geq1.
  \end{equation}
  Define the so-called derivative martingale by
\[
D_n:=\sum_{|u|=n}\theta^*(nx^*-S_u)e^{\theta^*(S_u-nx^*)},\quad n\geq 1.
\]
It is known from \cite{BK04} and \cite{A13} that under assumptions \reff{momoffspring}, \reff{expmom}, \reff{hyp3} and \reff{hyp4},  there exists a non-negative random variable $D_{\infty}$ such that
\[
D_{n}\overset{\P-a.s.}{\longrightarrow}D_{\infty},\quad \textrm{as }n\rar\infty,
\]
where $\{D_\infty>0\}=\{\cal T=\infty\}$ a.s.
Next, given  \eqref{momoffspring}, \eqref{expmom}, \eqref{hyp3} and \eqref{hyp4}, A\"id\'ekon \cite{A13} proved the convergence in law of $M_n-m_n$ as follows. For any $x\in\mbb R$,
\begin{equation}\label{cvgM}
\lim_{n\rightarrow\infty}\P(M_n\leq m_n+x)=\E[e^{-C e^{-x}D_\infty}],
\end{equation}
where $C>0$ is a constant.
In this work, we are going to study the asymptotic of $\P(M_n\leq m_n-\ell_n)$ for $1\ll\ell_n=O(n)$, as well as that of $\P(0<D_\infty<\varepsilon)$ which is closely related to $\P(M_n\leq m_n-\ell_n)$ by \eqref{cvgM}. Let us introduce the minimal offspring for $\cal T$:
\[
b:=\min\{k\geq0: p_k>0\}.
\]
We first present the main results in B\"ottcher case where $b\geq2$ and $\cal T=\infty$.
\begin{theorem}[B\"ottcher case, bounded step size]
\label{BDLD} Assume \eqref{momoffspring}, \eqref{expmom} and $b\geq 2$. Suppose that $\text{ess inf }X=-L$ for some $0<L<\infty$,
 then for $x\in (-L, x^*)$,
\begin{equation}\label{LDBd}
\lim_{n\rightarrow\infty}\frac{1}{n}\log\left[-\log\P\left(M_n\leq xn\right) \right] =\frac{x^*-x}{x^*+L} \log b.
\end{equation}
If $\P(X=-L)>0$, then (\ref{LDBd}) holds also for  $x=-L.$
\end{theorem}
\begin{remark}
Note that the assumptions \eqref{momoffspring} and \eqref{expmom} do not imply the second logarithmic order of $M_n$.
\end{remark}
\begin{theorem}[Bounded step size: moderate deviation]\label{BDMD}
Assume \eqref{momoffspring}, \eqref{expmom}, \eqref{hyp3}, \eqref{hyp4} and $b\geq2$. Suppose that $\text{ess inf }X=-L$ for some $0<L<\infty$. Then  for any positive increasing sequence $\ell_n$ such that $\ell_n\uparrow\infty$ and $\limsup_{n\rar\infty}\frac{\ell_n}{n}<x^*+L$,
\begin{equation}\label{MDbd}
\P\left(M_n\leq x^*n-\frac{3}{2\theta^*}\log n-\ell_n\right)=e^{-e^{\ell_n\beta (1+o_n(1))}}.
\end{equation}
where $\beta:=\frac{\log b}{x^*+L}\in(0,\theta^*)$ because of \eqref{hyp3}.
\end{theorem}
\begin{remark}\label{Hu01}
Hu \cite{Hu16} obtained this moderate deviation \eqref{MDbd} for  $\ell_n=o(\log n)$ in a more general setting with bounded step size  and without assuming independence between offsprings and motions. One could check that $\beta=\sup\{a>0:\P(\sum_{|u|=1}e^{-a(x^*-X_u)}\geq1)=1\}=\frac{\log b}{x^*+L}$ is coherent with that defined in (1.10) of \cite{Hu16}.
\end{remark}
\begin{remark}\label{BDMDS}
Suppose that all assumptions in Theorem \ref{BDMD} hold. Then by Theorem 1.3 in \cite{Hu16}, we have
\[
\P(D_\infty<\varepsilon)=e^{-\varepsilon^{-\frac{\beta}{\theta^*-\beta}+o(1)}}.
\]
\end{remark}
\begin{theorem}[B\"ottcher case, Weibull left tail]
\label{weibulltail}
Assume \eqref{momoffspring}, \eqref{expmom}, \eqref{hyp3}, \eqref{hyp4} and $b\geq2$. Suppose $\P(X\leq -z)=\Theta(1) e^{-\lambda z^{\alpha}}$ as $z\rightarrow+\infty$ for some constant $\alpha\geq1$ and $\lambda>0$. Then
 for any positive increasing sequence $\ell_n$ such that $\ell_n\uparrow\infty$ and $\ell_n=O(n)$,
\begin{equation}\label{MDWei}
\lim_{n\rightarrow\infty}\frac{1}{\ell_n^{\alpha}}\log\P\left(M_n\leq m_n- \ell_n\right) =  -\lambda\l(b^{\frac{1}{\alpha-1}}-1\r)^{\alpha-1}.
\end{equation}
where for convenience, $\l(b^{\frac{1}{\alpha-1}}-1\r)^{\alpha-1}:=b$ for $\alpha=1$. In particular, for any $x<x^*$,
\begin{equation}\label{LDWei}
\lim_{n\rightarrow\infty}\frac{1}{n^{\alpha}}\log\P\left(M_n\leq xn\right) = -\lambda\l(b^{\frac{1}{\alpha-1}}-1\r)^{\alpha-1}
  (x^*-x)^{\alpha}.
\end{equation}
\end{theorem}
\begin{remark}
Note that if $\alpha<1$, the assumption \eqref{expmom} can not be satisfied and we are in another regime where $M_n$ grows faster than linear in time; see \cite{G00}.
\end{remark}
The weak convergence \eqref{cvgM} shows that $\P(M_n\leq m_n-\ell_n)$ and $\P(D<\ez)$ are closely related. So inspired by Theorem \ref{weibulltail}, one obtains the following result.
\begin{proposition}[B\"ottcher case, Weibull left tail]\label{weibulltailleft}
Suppose that all assumptions in Theorem \ref{weibulltail} hold. Then
\begin{eqnarray}\label{Dweib}
\lim_{\varepsilon\rar0+}\frac{1}{(-\log\varepsilon)^{\alpha} }\log\P(D_{\infty}<\varepsilon)=-\frac{ \lz }{\l(\theta^*\r)^{\alpha} } \l(b^{\frac{1}{\alpha-1}}-1\r)^{\alpha-1}.
\end{eqnarray}
\end{proposition}
\begin{theorem}[B\"ottcher case, Gumbel left tail]
\label{gumbeltail} Assume \eqref{momoffspring}, \eqref{expmom}, \eqref{hyp3}, \eqref{hyp4} and $b\geq2$.
Suppose $\P(X\leq -z)=\Theta(1)\exp(-e^{ z^{\alpha}})$ as $z\rightarrow+\infty$ for some constant $\alpha>0$. Then
for any positive increasing sequence $\ell_n$ such that $\ell_n\uparrow\infty$ and $\ell_n=O(n)$,
\begin{equation}\label{MDGum}
\lim_{n\rightarrow\infty}\ell_n^{-\frac{\alpha}{\alpha+1}}\log\l[-\log\P\left(M_n\leq m_n- \ell_n\right)\r] =  \left(\frac{1+\alpha}{\alpha} \log b\right)^{\frac{\alpha}{\alpha+1}}.
\end{equation}
In particular, for any $x<x^*$,
\begin{equation}\label{LDGum}
\lim_{n\rightarrow\infty}{n^{-\frac{\alpha}{\alpha+1}}}\log\left[-\log
\P\left(M_n\leq xn\right)\right]
= \left(\frac{1+\alpha}{\alpha}\log b\right)^{\frac{\alpha}{\alpha+1}}(x^*-x)^{\frac{\alpha}{\alpha+1}} .
\end{equation}
\end{theorem}
Again, inspired by Theorem \ref{gumbeltail} and the weak convergence \eqref{cvgM}, we have the following result.
\begin{proposition}[B\"ottcher case, Gumbel left tail]
\label{gumbeltailleft} Suppose that all assumptions in Theorem \ref{gumbeltail} hold. Then
\begin{eqnarray}\label{Dgumb}
\lim_{\varepsilon\rar0+}\frac{1}{(-\log\varepsilon)^{\frac{\az}{\az+1}} }\log\l[-\log\P(D_{\infty}<\varepsilon)\r]= \left(\frac{1+\alpha}{\theta^*\alpha} \log b\right)^{\frac{\alpha}{\alpha+1}}.
\end{eqnarray}

\end{proposition}
Next theorem concerns the Schr\"oder case where $p_0+p_1>0$. Let $q:=\P({\cal T}<\infty)\in [0,1)$ be the extinction probability and $f(s):=\sum_{k\geq0}p_ks^k$, $s\in[0,1]$ be the generating function of offspring. Let $\P^s(\cdot):=\P(\cdot\vert \mathcal{T}=\infty)$. Denote $\max\{a, 0\}$ by $a_+$ for any real number $a\in\mbb R$.
\begin{theorem}[Schr\"oder case]
\label{Schroder}
Assume \eqref{momoffspring}, \eqref{expmom}, \eqref{hyp3}, \eqref{hyp4} and $0<p_0+p_1<1$. Then for any positive sequence $(\ell_n)$ such that $\ell_n\uparrow\infty$ and that $\ell^*:=\lim_{n\rightarrow\infty}\frac{\ell_n}{n}$ exists with $\ell^*\in[0,\infty)$, we have
\begin{equation}\label{MDS}
\lim_{n\rightarrow\infty}\frac{1}{\ell_n}\log \P^s\l(M_n\leq m_n-\ell_n\r)=H(x^*, \gamma),
\end{equation}
where $\gamma=\log f'(q)$ and
\begin{eqnarray} H(x^*, \gamma)=\sup_{y\geq (x^*-\ell^*)_+}\frac{\gamma-I(x^*-\ell^*-y)}{\ell^*+y}=\sup_{a\geq \ell^*}\frac{\gamma-I(x^*-a)}{a}.
\end{eqnarray} In particular, we have for any $x<x^*$,
\begin{equation}\label{LDS}
\lim_{n\rightarrow\infty}\frac{1}{n}\log\P^s\left(M_n\leq x n\right) = (x^*-x)\sup_{a\leq x}\frac{-I(a)+\gz}{x^*-a}.
\end{equation}
\end{theorem}
\begin{remark}
(\ref{LDS}) was obtained first by Gantert and H\"ofelsauer in \cite{GH18}. In fact, it is shown in \cite{GH18} that for any $x<x^*$,
$$\lim_{n\rightarrow\infty}\frac{1}{n}\log\P^s\left(M_n\leq x n\right) =-\inf_{t\in(0,1]}\{-t\gz+tI((x-(1-t)x^*)/t)\}.$$
Then one can check that
\[-\inf_{t\in(0,1]}\{-t\gamma+tI((x-(1-t)x^*)/t)\}=(x^*-x)\sup_{a\leq x}\frac{-I(a)+\gz}{x^*-a}.\]
\end{remark}
\begin{remark}\label{Hu02}
When $\ell_n=o(\log n)$, \reff{MDS} was obtained by Hu in \cite{Hu16} in a more general framework. 
In fact, if restricted to our setting, then conditions (1.5) and (1.6) in \cite{Hu16} is equivalent to say that there exists a constant $t^*>0$ such that
$$
\log f'(q)+t^*x^*+\log \psi(-t^*)=0,\textrm{ and }\psi(-t)<\infty \textrm{ for some }t>t^*.
$$
Since $\ell_n=o(\log n)$, then $\ell^*=0.$ So conditions (1.5) and (1.6) in \cite{Hu16} make sure that $a^*:=x^*-(\log \psi(t))'\mid_{t=-t^*}$ is exactly the $\arg\max$ of $a\mapsto\frac{\gamma-I(x^*-a)}{a}$ on $[0,\infty)$; i.e.;

$$
\frac{\gamma-I(x^*-a^*)}{a^*}=\sup_{a\geq 0}\frac{\gamma-I(x^*-a)}{a}=t^*.
$$
\end{remark}
\begin{remark}\label{Hu02a}
If all assumptions in Theorem \ref{Schroder} hold, then by Theorem 1.3 in \cite{Hu16}, we have
\[
\P(0<D_{\infty}<\varepsilon)\asymp \varepsilon^{t^*},\quad \textrm{as }\varepsilon\rar0.
\]
\end{remark}
\paragraph{General strategy:}  Let us explain our main ideas here, especially  for  $\P(M_n\leq m_n-\ell_n)$ in B\"ottcher case. Intuitively, to get an unusually lower maximum, we need to control both the size of the genealogical tree and the displacements of individuals. More precisely, we need that at the very beginning, the size of the genealogical tree is small with all individuals moving to some atypically lower place. So, we take some intermediate time $t_n$ and suppose that the genealogical tree is $b$-regular up to time $t_n$ and that all individuals at time $t_n$ are located below certain ``critical" position $-c_n$. Then the system continues with $b^{t_n}$ i.i.d. branching random walks started from places below $-c_n$. By choosing $t_n$ and $c_n=\Theta(\ell_n)$ in an appropriate way, we can expect that the maximum at time $n$ stays below $m_n-\ell_n$ with high probability.

Note that, the time $t_n$ varies in different cases. If the step size is bounded from below, $t_n=\Theta(\ell_n)$. If the step size has Weibull tail or Gumbel tail, $t_n=o(\ell_n)$.

Our arguments and techniques are also inspired by \cite{CH17} where we studied the large deviation of empirical distribution of branching random walk. All these ideas work also for studying the small ball probability of $D_\infty$.

The rest of this paper is organized as follows. We treat the cases with bounded step size in Section 2. Then, Section 3 proves Theorems \ref{weibulltail} and \ref{gumbeltail}, concerning the cases with unbounded step size. In Section \ref{lefttailD}, we study $\P(0<D_\infty<\varepsilon)$ and prove Propositions \ref{weibulltailleft} and \ref{gumbeltailleft}. Finally, we prove Theorems \ref{Schroder} for Schr\"oder case in Section 5.

Let $C_1, C_2, \cdots$ and $c_1, c_2, \cdots$ denote positive constants whose values may change from line to line.

\section{B\"ottcher case with step size bounded on the left side: Proofs of Theorems \ref{BDLD}, \ref{BDMD}:}

In this section, we always suppose that $b\geq2$ and $\text{ess inf }X=-L$ with $L\in(0,\infty)$. Assumption \eqref{expmom} yields that
$M_n=x^*n+o(n)$ with $x^*\in(0,\infty)$. We are going to prove that for any $-L<x<x^*$,
\begin{equation}\label{LDBdPf}
\P(M_n\leq xn)=e^{-e^{(1+o(1))\beta(x^*-x)n}},\quad \textrm{as } n\rightarrow\infty,
\end{equation}
with $\beta=\frac{\log b}{x^*+L}$. Next, for the second order of $M_n$, there are several regimes. We assume \eqref{hyp3} and \eqref{hyp4} to get the classical one: $M_n=m_n+o_\P(\log n)$ with $m_n=x^*n-\frac{3}{2\theta^*}\log n$. In this regime, we are going to prove that for any positive sequence $\ell_n\uparrow\infty$  such that $\limsup_{n\rar\infty}\frac{\ell_n}{n}<x^*+L$,
\begin{equation}\label{MDBdPf}
\P(M_n\leq m_n-\ell_n)=e^{-e^{(1+o(1))\beta\ell_n}}, \quad\textrm{as } n\rightarrow\infty.
\end{equation}
The proofs of \eqref{LDBdPf} and \eqref{MDBdPf} basically follow the same ideas. But \eqref{LDBdPf} needs to be treated in a more general regime, without second order estimates.

For later use, let us introduce the counting measures as follows: for any $B\subset \mbb R$,
\[
Z_n(B):=\sum_{|u|=n}1_{S_u\in B}, \quad\forall n\geq0.
\]
For simplicity, we write $Z_n$ for $Z_n(\mbb R)$ to represent the total population of the $n$-th generation. It is clear that $Z_n\geq b^n$. For any $u\in\cal T$, let
\[
M^u_n:=\max_{|z|=n+|u|, z\geq u}\{S_z-S_u\},\quad \forall n\geq0,
\]
be the maximal relative position of descendants of $u$. Clearly, $(M^u_n)_{n\geq0}$ is distributed as $(M_n)_{n\geq0}$.
\subsection{Proof of Theorem \ref{BDLD}}

In this section, we show that for any $x\in (-L, x^*)$, \eqref{LDBdPf} holds. We use $t_n^-$ to denote the intermediate time chosen for the lower bounds and $t^+_n$ for upper bounds.

\subsubsection{Lower bound of Theorem \ref{BDLD}}
\noindent  As $x>-L$, let $L':=L-\eta$ with any sufficiently small $\eta>0$ such that $x>-L+\eta$. Notice that $\text{ess inf }X=-L$ implies that $\P(X\leq -L')>0$ for any $\eta>0$.  Observe that for some intermediate time $t_n^-$, whose value will be determined later, if we let every individual before the $t_n^-$-th generation make a displacement less than $-L'$, then
\begin{align*}
\P(M_n\leq xn)\geq&\P(Z_{t_n^-}=b^{t_n^-}; \forall |u|=t_n^-, S_u\leq -L't_n^-; M_n\leq xn )\\
\geq & \P(Z_{t_n^-}=b^{t_n^-}; \forall |u|=t_n^-, S_u\leq -L't_n^-; \max_{|u|=t_n^-}M^u_{n-t_n^-}\leq xn+L't_n^- ),
\end{align*}
where $\{M^u_{n-t_n^-}\}$ are i.i.d. copies of $M_{n-t_n^-}$. By Markov property at time $t_n^-$, one gets that
\begin{align}\label{lowerbd1}
\P(M_n\leq xn)\geq&\P(Z_{t_n^-}=b^{t_n^-}; \forall |u|=t_n^-, S_u\leq -L't_n^-) \P(M_{n-t_n^-}\leq xn+L't_n^-)^{b^{t_n^-}}\cr
\geq & \P(Z_{t_n^-}=b^{t_n^-}; \forall 1\leq |u|\leq t_n^-, X_u \leq-L') \P(M_{n-t_n^-}\leq xn+Lt_n^-)^{b^{t_n^-}}\cr
=&p_b^{\sum_{k=0}^{t_n^--1}b^k}\P(X\leq-L')^{\sum_{k=1}^{t_n^-}b^k}\P(M_{n-t_n^-}\leq xn+L't_n^-)^{b^{t_n^-}}.
\end{align}
Next, we shall estimate $\P(M_{n-t_n^-}\leq xn+L't_n^-)^{b^{t_n^-}}$. The sequel of this proof will be divided into two subparts depending on whether $x^*=R:=\text{ess sup }X$ or not, respectively.

{\it Subpart 1: the case  $x^*=R$.} Note that we have $R<\infty$ now. Take $t_n^-=\lceil\frac{(R-x)n}{R+L'}\rceil$ so that $xn+L't_n^-\geq R(n-t_n^-)$. Thus,
\[
\P(M_{n-t_n^-}\leq xn+L't_n^-)^{b^{t_n^-}}=1.
\]
Going back to \eqref{lowerbd1}, one sees that for some $C\in \mbb R_+^*$,
\begin{equation}\label{lowerbd2}
\P(M_n\leq xn)\geq p_b^{\frac{b^{t_n^-}-1}{b-1}} \l(\P(X\leq -L')\r) ^{\frac{b^{t_n^-+1}-b}{b-1}}\geq e^{-C b^{t_n^-}}.
\end{equation}
It follows readily that for any $x\in (-L, x^*)$,
\begin{equation}\label{lowerbd3}
\limsup_{n\rightarrow\infty}\frac{1}{n}\log[-\log\P(M_n\leq xn)]\leq \frac{x^*-x}{x^*+L-\eta}\log b.
\end{equation}
Letting $\eta\downarrow0$ yields what we need.

{\it Subpart 2: the case $x^*<R\in(0,\infty]$.} Now we have $I(x^*)=\log m$ because $I$ is finite and continuous in $(0,R)$.
Moreover, $I(x)<\infty$ for some $x>x^*$. For any sufficiently small $a>0$, one has
\[
\log m<I(x^*+a)<\infty, \textrm{ and } \lim_{a\downarrow0}I(x^*+a)=I(x^*)=\log m.
\]
Recall that $-x<L'$. Let
$t=\frac{x^*+a-x}{x^*+L'+a}$ and $t_n^-=\lceil tn\rceil$ so that
 ${xn+L't_n^-}>(x^*+a)(n-t_n^-)\gg1$ for all $n$ large enough.  Therefore,
\begin{align*}
\P(M_{n-t_n^-}\leq xn+L't_n^-)^{b^{t_n^-}} &\geq \l(1-\P(M_{n-t_n^-}>(x^*+a)(n-t_n^-))\r)^{b^{t_n^-}}.
\end{align*}
Here we apply the large deviation result obtained in \cite{GH18}. More precisely, as the maximum of independent random walks dominates stochastically $M_n$, one has
\[
\P(M_{n-t_n^-}>(x^*+a)(n-t_n^-))\leq \E[Z_{n-t_n^-}]\P(S_{n-t_n^-}\geq (x^*+a)(n-t_n^-))\leq e^{-\l(I(x^*+a)-\log m\r)(n-t_n^-)}
\]
which yields that
\[
\P(M_{n-t_n^-}\leq xn+L't_n^-)^{b^{t_n^-}} \geq \l(1-e^{-\l(I(x^*+a)-\log m\r)(n-t_n^-)}\r)^{b^{t_n^-}}.
\]
Note that $\log (1-x)\geq -2x$ for any $x\in[0,1/2]$. Let $\delta(a):=I(x^*+a)-\log m$. Then for all sufficiently large $n\geq1$,
\[
\P(M_{n-t_n^-}\leq xn+L't_n^-)^{b^{t_n^-}} \geq e^{-2 e^{-\delta(a)(n-t_n^-)}b^{t_n^-}}.
\]
Plugging this into \eqref{lowerbd1} implies that
\begin{equation}\label{lowerbd5}
\P(M_n\leq xn)\geq e^{-Cb^{t_n^-}}e^{-2 e^{-\delta(a)(n-t_n^-)}b^{t_n^-}}.
\end{equation}
Thus we have
\begin{equation}\label{lowerbd7}
\limsup_{n\rightarrow\infty}\frac{1}{n}\log[-\log\P(M_n\leq xn)]\leq t\log b.
\end{equation}
Since $I(x^*)=\log m$,  letting $a\downarrow0$ (hence $t\downarrow\frac{x^*-x}{x^*+L'}$  and $\delta(a)\downarrow0$) gives \[
\limsup_{n\rightarrow\infty}\frac{1}{n}\log\left[-\log\P\left(M_n\leq xn\right) \right] \leq \frac{x^*-x}{x^*+L-\eta} \log b,
\]
which implies the desired lower bound because $\eta$ is arbitrary small. \qed

\subsubsection{Upper bound of Theorem \ref{BDLD}}
In this section, we show that
\[
\P(M_n\leq xn)\leq e^{- b^{\frac{(x^*-x)n}{x^*+L}+o(n)}}.
\]
Note that for any $1\leq t_n^+\leq n$, $Z_{t_n^+}(\cdot)$ is supported by $[-Lt_n^+,\infty)$ a.s. Moreover, $Z_{t_n^+}\geq b^{t_n^+}$. Observe that
\begin{align}\label{upbd4}
\P(M_n\leq xn)\leq  &\P\l(Z_{t_n^+}([-Lt_n^+,\infty))\geq  b^{t_n^+}; M_n\leq xn\r)\nonumber\\
\leq &\P\l(Z_{t_n^+}([-Lt_n^+,\infty))\geq  b^{t_n^+}; \max_{|u|=t_n^+; S_u\geq -Lt_n^+}(S_u+M_{n-t_n^+}^u)\leq xn\r)\nonumber\\
\leq & \P\l(M_{n-t_n^+}\leq xn+Lt_n^+\r)^{ b^{t_n^+}}.
\end{align}

It remains to estimate $\P(M_{n-t_n^+}\leq xn+Lt_n^+)^{ b^{t_n^+}}$. Again, the proof will  be divided into two subparts.

{\it Subpart 1: the case $x^*=R$.} By taking $t_n^+=\lfloor \frac{(R-x)n}{R+L}\rfloor-1$ so that $xn+Lt_n^+< R(n-t_n^+)$, one has
\begin{align}\label{upbdsp1}
\P(M_{n-t_n^+}\leq xn+Lt_n^+)^{ b^{t_n^+}}\leq &\P(M_{n-t_n^+}<R(n-t_n^+))^{b^{t_n^+}}\cr
= &\l(1-\P(M_{n-t_n^+}\geq R(n-t_n^+))\r)^{ b^{t_n^+}}\nonumber\\
\leq &\l(1-\frac{c}{n-t_n^+}\r)^{ b^{t_n^+}}\leq e^{-c\frac{ b^{t_n^+}}{(n-t_n^+)}},
\end{align}
where we use the fact that $\P(M_N\geq RN )\geq c/N$ for some $c\in (0,1)$ and all $N\geq1$. In fact, we could construct a Galton-Watson tree with offspring $\sum_{|u|=1}1_{X_u=R}$. Here $\E[\sum_{|u|=1}1_{X_u=R}]=m\P(X=R)\geq 1$ since $x^*=R$. Its survival probability is positive if $\E[\sum_{|u|=1}1_{X_u=R}]>1$. Even when $\E[\sum_{|u|=1}1_{X_u=R}]=1$, it is critical and the survival probability up to generation $N$ is larger than $c/N$ for some $c>0$ and for all $N\geq1$. In fact, its survival up to generation $N$ implies that some individual at time $N$ has position $RN$. So, $\P(M_N\geq RN )\geq c/N$. We hence conclude from \reff{upbd4} and \eqref{upbdsp1} that
\[
\liminf_{n\rightarrow\infty}\frac{1}{n}\log[-\log\P(M_n\leq xn)]\geq \frac{(R-x)\log b}{R+L}.
\]

{\it Subpart 2: the case $x^*<R$.} First recall a result from \cite{GH18}( see Theorem 3.2) which says that
\begin{equation}\label{GH18}
\lim_{n\rar\infty}\frac{1}{n}\log \P(M_n> xn)=\log m- I(x),\quad \textrm{ for } x>x^*.
\end{equation}
So for any sufficiently small $a>0$ such that $\delta(a)=I(x^*+a)-\log m\in(0,\infty)$, for any $x>-L$, let $t=\frac{x^*+a-x}{L+x^*+a}\in(0,1)$ and $t_n^+=\lfloor t n\rfloor$ so that
$x^*<\frac{xn+Lt_n^+}{n-t_n^+}\leq x^*+a$.
Then  for all $n$ large enough,
\begin{align}\label{upbd5}
\P(M_{n-t_n^+}\leq xn+Lt_n^+)^{ b^{t_n^+}} &=\l(1-\P\l(M_{n-t_n^+}>\frac{xn+Lt_n^+}{n-t_n^+}(n-t_n^+)\r)\r)^{ b^{t_n^+}}\nonumber\\
&\leq  \l(1-\P\l(M_{n-t_n^+}> (x^*+a)(n-t_n^+)\r)\r)^{ b^{t_n^+}}\cr
&\leq \l(1-\exp\l\{ -\l(I\l( x^*+a\r)-\log m+{\delta(a)}\r)(n-t_n^+)\r\}\r)^{ b^{t_n^+}}\cr
&\leq e^{- e^{-2\delta(a)(n-t_n^+)} b^{t_n^+}},
\end{align}
where the second inequality follows from \reff{GH18}. Plugging (\ref{upbd5}) into (\ref{upbd4}) yields that
\[
\liminf_{n\rightarrow\infty}\frac{1}{n}\log[-\log\P(M_n\leq xn)]\geq -2\delta(a)(1-t)+t\log b.
\]
Again  letting $a\downarrow0$  (hence $\delta(a)\downarrow 0$ and $t\downarrow\frac{x^*-x}{x^*+L}$) gives the desired upper bound.

If $\P(X=-L)>0$, then the arguments for lower bound work well for $x=-L$ and $L'=L$. For the upper bound, it is easy to see that all displacements are $-L$ up to the $n$-th generation. We thus could also obtain \reff{LDBd} for $x=-L$. \qed

\subsection{Proof of Theorem \ref{BDMD}}

From now on, we assume \eqref{hyp3} and \eqref{hyp4} so that $M_n=m_n+o_\P(\log n)$. Moreover, it is known in \cite{A13} that $M_n-m_n$ converges in law to some random variable on the survival of $\cal T$. In fact, \eqref{hyp4} is slightly stronger than the conditions given in \cite{A13}. Because of this convergence in law in B\"ottcher case, we can find some $y^*\in\mbb R_+$ so that
\begin{equation}\label{choiceofy}
\P(M_n\leq m_n-y^*)\leq 1/2 \leq \P(M_n \leq m_n+y^*).
\end{equation}

Now we are ready to prove that for any increasing sequence $\ell_n=O(n)$ such that $\ell_n\uparrow\infty$ and $\limsup_{n\rar\infty}\frac{\ell_n}{n}<x^*+L$,
\begin{equation}\label{MDbdeq}
\P\left(M_n\leq m_n-\ell_n\right)=e^{-e^{\ell_n\beta (1+o_n(1))}},
\end{equation}
where $\beta=\frac{\log b}{x^*+L}$ and $m_n=x^*n-\frac{3}{2\theta^*}\log n$.
\subsubsection{Lower bound of Theorem \ref{BDMD}}
Similarly to the previous section on large deviation, let us again take some intermediate time $t_n^-\in [1,n-1]$ and $L'=L-\eta$ with $\eta>0$,
\begin{align*}
\P(M_n\leq m_n-\ell_n)\geq & \P\l(Z_{t_n^-}=b^{t_n^-}; \forall |u|\leq t_n^-, X_u\leq -L'; M_n\leq m_n-\ell_n\r)\\
\geq & \P\l(Z_{t_n^-}=b^{t_n^-}; \forall |u|\leq t_n^-, X_u\leq -L'; \max_{|v|=t_n^-}M^v_{n-t_n^-}\leq m_n-\ell_n+L't_n^-\r),
\end{align*}
which by branching property is larger than
\begin{align*}
\P\l(Z_{t_n^-}=b^{t_n^-}; \forall |u|\leq t_n^-, X_u\leq -L'\r) \P(M_{n-t_n^-}\leq m_n-\ell_n+L't_n^-)^{b^{t_n^-}}.
\end{align*}
Here we choose $t_n^-=\lceil\frac{\ell_n+K_0}{L'+x^*}\rceil$ with $K_0\geq1$ a fixed large constant so that $m_n-\ell_n+L't_n^-\geq m_{n-t_n^-}+y^*$. Consequently,
\begin{align*}
\P(M_n\leq m_n-\ell_n)\geq & \P\l(Z_{t_n^-}=b^{t_n^-}; \forall |u|\leq t_n^-, X_u\leq -L'\r) \P(M_{n-t_n^-}\leq m_{n-t_n^-}+y^*)^{b^{t_n^-}}\\
\geq & p_b^{\sum_{k=0}^{t_n^--1}b^k}\P(X\leq -L')^{\sum_{k=1}^{t_n^-}b^k}\P(M_{n-t_n^-}\leq m_{n-t_n^-}+y^*)^{b^{t_n^-}},
\end{align*}
where the last inequality holds because of the independence between offsprings and motions. Now note that $-L=\text{ ess inf } X$ means that $q_L:=\P(X\leq-L')\in(0,1)$. By \eqref{choiceofy},
\begin{align*}
\P(M_n\leq m_n-\ell_n)\geq & p_b^{\sum_{k=0}^{t_n^--1}b^k}q_L^{\sum_{k=1}^{t_n^-}b^k}(1/2)^{b^{t_n^-}}=e^{-\Theta(b^{t_n^-})},
\end{align*}
with $t_n^-=\frac{\ell_n+K_0}{L+x^*-\eta}$. Letting $n\rightarrow\infty$ then $\eta\rightarrow0$ gives that
\[
\limsup_{n\rightarrow\infty}\frac{1}{\ell_n}\log[-\log\P(M_n\leq m_n-\ell_n)]\leq \frac{\log b}{L+x^*}.
\]
\subsubsection{Upper bound of Theorem \ref{BDMD}}
Let $B_n=[-Lt_n^+,\infty)$ with some intermediate time $t_n^+$ to be determined later. Observe that
\begin{align*}
\P(M_n\leq m_n-\ell_n)= &\P\l(Z_{t_n^+}(B_n)\geq b^{t_n^+}; M_n\leq m_n-\ell_n\r)\\
= & \P\l(Z_{t_n^+}(B_n)\geq  b^{t_n^+}; \max_{|u|=t_n^+, S_u\in B_n}(S_u+M^u_{n-t_n^+})\leq m_n-\ell_n\r)\\
\leq & \P\l(Z_{t_n^+}(B_n)\geq  b^{t_n^+}; \max_{|u|=t_n^+, S_u\in B_n}M^u_{n-t_n^+}\leq m_n-\ell_n+Lt_n^+\r)
\end{align*}
Let $t_n^+:=\lfloor\frac{\ell_n-y^*}{L+x^*}\rfloor$ so that $m_n-\ell_n+Lt_n^+\leq m_{n-t_n^+}-y^*$. Then by \eqref{choiceofy},
\begin{align}
\P\l(Z_{t_n^+}(B_n)\geq  b^{t_n^+}; \max_{|u|=t_n^+, S_u\in B_n}M^u_{n-t_n^+}\leq m_n-\ell_n+Lt_n^+\r)\leq& \P(M_{n-t_n^+}\leq m_{n-t_n^+}-y)^{ b^{t_n^+}}\nonumber\\
\leq& (1/2)^{ b^{t_n^+}}.
\end{align}
We hence obtain that
\[
\P(M_n\leq m_n-\ell_n)\leq e^{-c b^{t_n^+}},
\]
with $b^{t_n^+}=\Theta(e^{\beta\ell_n})$. This suffices to conclude Theorem \ref{BDMD}.
\section{B\"ottcher case with step size of (super)-exponential left tail}
\subsection{Proof of Theorem \ref{weibulltail}: step size of Weibull tail}

Given Weibull tail distribution for the step size $X$, we are going to prove that, for any increasing sequence $(\ell_n)$ such that $\ell_n\leq O(n)$ and $\ell_n\uparrow\infty$, one has
\begin{equation}
\lim_{n\rightarrow\infty}\frac{1}{\ell_n^\alpha}\log\P\left(M_n\leq x^*n -\frac{3}{2\theta^*}\log n -\ell_n\right) =  -\lambda\l(b^{\frac{1}{\alpha-1}}-1\r)^{\alpha-1},
\end{equation}
where $\l(b^{\frac{1}{\alpha-1}}-1\r)^{\alpha-1}=b$ for $\alpha=1$.

\subsubsection{Lower bound of Theorem \ref{weibulltail}}
\label{lowerbdweibull}
\paragraph{The case $\alpha=1$}

In this case, we could show that
\[
\P(M_n\leq m_n-\ell_n)\geq e^{-\lambda \ell_n b}.
\]
In fact, at the first generation, we suppose that there are exactly $b$ individuals and that all of them are located below $-(\ell_n+x^*+y^*)$. So, as $m_n-\ell_n+(\ell_n+x^*+y^*)\geq m_{n-1}+y^*$,
\begin{align*}
\P(M_n\leq m_n-\ell_n)\geq &\P(Z_1=b; \forall |u|=1, X_u\leq -(\ell_n+x^*+y^*); M_n\leq m_n-\ell_n)\\
= &\P(Z_1=b; \forall |u|=1, X_u\leq -(\ell_n+x^*+y^*); \max_{|u|=1}(X_u+M^u_{n-1})\leq m_n-\ell_n)\\
\geq &\P\l(Z_1=b; \forall |u|=1, X_u\leq -(\ell_n+x^*+y^*); \max_{|u|=1}(M^u_{n-1})\leq m_{n-1}+y^*\r).
\end{align*}
By Markov property, this implies that
\begin{align*}
\P(M_n\leq m_n-\ell_n)\geq &\P\l(Z_1=b; \forall |u|=1, X_u\leq -(\ell_n+x^*+y^*)\r)\P\l(M_{n-1}\leq m_{n-1}+y^*\r)^b\\
=& p_b \P(X\leq -(\ell_n+x^*+y^*))^b\P\l(M_{n-1}\leq m_{n-1}+y^*\r)^b,
\end{align*}
where $\P(X\leq -(\ell_n+x^*+y^*))=\Theta(1)e^{-\lambda \ell_n}$ and $\P\l(M_{n-1}\leq m_n+y^*\r)\geq 1/2$. Consequently,
\[
\P(M_n\leq m_n-\ell_n)\geq \Theta(1) e^{-\lambda \ell_n b}.
\]
\paragraph{The case $\alpha>1$}

We prove here that
\[
\liminf_{n\rightarrow\infty}\frac{1}{\ell_n^\alpha}\log\P\left(M_n\leq m_n-\ell_n\right) \geq -\lambda\l(b^{\frac{1}{\alpha-1}}-1\r)^{\alpha-1}.
\]

By the assumption of Theorem \ref{weibulltail}, there exist two constants $0<c<1$ and $0<C<\infty$ such that for any $x>0$,
\begin{equation}\label{supexptail}
ce^{-\lambda x^\alpha}\leq \P(X\leq -x) \leq C e^{-\lambda x^\alpha}.
\end{equation}

We choose $t_n^-=o(\ell_n)$ such that $t_n^-\uparrow\infty$ and suppose that up to the $t_n^-$-th generation, the genealogical tree is a $b$-regular tree. For any $|u|=k$ with $1\leq k\leq t_n^-$, we suppose that its displacement $X_u$ is less than $- a_k $ with some $a_k>0$. We will determine the sequence $(a_k)_{k\geq1}$ later. Therefore,
\begin{align*}
&\P(M_n\leq m_n-\ell_n)\geq \P\l(Z_{t_n^-}=b^{t_n^-}; \forall |u|=k\in \{1,\cdots, t_n^-\}, X_u<-a_k ; M_n\leq m_n-\ell_n \r)\\
\geq &\P\l(Z_{t_n^-}=b^{t_n^-}; \forall |u|=k\in \{1,\cdots, t_n^-\}, X_u<-a_k; \max_{|z|=t_n^-}(S_z+M_{n-t_n^-}^z)\leq m_n-\ell_n \r)\\
\geq & \P\l(Z_{t_n^-}=b^{t_n^-}; \forall |u|=k\in \{1,\cdots, t_n^-\}, X_u<-a_k; \max_{|z|=t_n^-}(M_{n-t_n^-}^z)\leq m_n-\ell_n+\sum_{k=1}^{t_n^-} a_k  \r).
\end{align*}
Once again by Markov property,  one has
\begin{multline}\label{cutattn}
\P(M_n\leq m_n-\ell_n)\\
\geq  \P\l(Z_{t_n^-}=b^{t_n^-}; \forall |u|=k\in \{1,\cdots, t_n^-\}, X_u<-a_k \r)\P\l(M_{n-t_n^-}\leq m_n-\ell_n+\sum_{k=1}^{t_n^-} a_k\r)^{b^{t_n^-}}.
\end{multline}
For the first term on the right hand side, by independence of branching structure and displacements,
\begin{align*}
&\P\l(Z_{t_n^-}=b^{t_n^-}; \forall |u|=k\in \{1,\cdots, t_n^-\}, X_u<-a_k \r)\\
=& p_b^{\sum_{k=0}^{t_n^--1}b^k}\prod_{k=1}^{t_n^-} \P(X<-a_k )^{b^k},
\end{align*}
which by \eqref{supexptail}, is larger than
\begin{align}\label{probefore}
p_b^{\frac{b^{t_n^-}-1}{b-1}} \prod_{k=1}^{t_n^-} c^{b^k}e^{-\lambda (a_k)^\alpha b^k}
= p_b^{\frac{b^{t_n^-}-1}{b-1}} c^{\frac{b^{t_n^-+1}-b}{b-1}}\exp\{-\lambda  \sum_{k=1}^{t_n^-}a_k^{\alpha} b^k\}.
\end{align}
Now, we take the values of $a_k$. Let $b_\alpha:=b^{\frac{1}{\alpha-1}}$ and $a_k=\frac{ (b_\alpha-1)}{b_\alpha^k}\ell_n$. Note that $\sum_{k=1}^{t_n^-}a_k=(1-b_\alpha^{-t_n^-})\ell_n$. Take $t_n^-=(\alpha-1)\frac{\log \ell_n}{\log b}$ so that for $n$ large enough,
\begin{equation}\label{choixa}
m_n-\ell_n+\sum_{k=1}^{t_n^-} a_k =m_n-\ell_n+ (1-b_\alpha^{-t_n^-})\ell_n \geq m_{n-t_n^-}+y^*.
\end{equation}
Meanwhile, one obtains that
\[
b^{t_n^-}=\ell_n^{\alpha-1}, \textrm{ and } \sum_{k=1}^{t_n^-}a_k^\alpha b^k=\ell_n^\alpha (b_\alpha-1)^{\alpha-1} (1-b_\alpha^{-t_n^-}).
\]
Plugging them into \eqref{probefore} yields that
\begin{eqnarray}\label{probefore01}
\P\l(Z_{t_n^-}=b^{t_n^-}; \forall |u|=k\in \{1,\cdots, t_n^-\}, X_u<-a_k \r)
\geq \exp\{- \lambda \ell_n^\alpha (b_\alpha-1)^{\alpha-1}-\Theta(\ell_n^{\alpha-1})\}.
\end{eqnarray}
Applying it and \eqref{choixa} to \eqref{cutattn} yields that
\begin{align*}
\P(M_n\leq m_n-\ell_n)\geq &\exp\{- \lambda \ell_n^\alpha (b_\alpha-1)^{\alpha-1}- \Theta( \ell_n^{\alpha-1})\}\P\l(M_{n-t_n^-}\leq m_{n-t_n^-}+y^*\r)^{b^{t_n^-}}\\
\geq & \exp\{- \lambda \ell_n^\alpha (b_\alpha-1)^{\alpha-1}- \Theta( \ell_n^{\alpha-1})\}(1/2)^{\ell_n^{\alpha-1}}.
\end{align*}
As a result,
\begin{equation}\label{Flowerbdweibull}
\P(M_n\leq m_n-\ell_n)\geq \exp\{-\lambda \ell_n^\alpha (b_\alpha-1)^{\alpha-1}- \Theta( \ell_n^{\alpha-1})\}.
\end{equation}
\subsubsection{Upper bound of Theorem \ref{weibulltail}}
\label{upbdWeibull}

In this section, we consider the upper bound of $\P(M_n\leq m_n-\ell_n)$. First we state the following lemma which gives a rough upper bound.

\begin{lem}\label{rupbd1}
Assume \eqref{momoffspring}, \eqref{expmom}, \eqref{hyp3}, \eqref{hyp4} and $b\geq2$. For any $\theta>0$ such that $\E[e^{-\theta X}]<\infty$ and for $n$ sufficiently large, we have
\begin{equation}\label{roughupbd1}
 \P(M_n\leq m_n-\ell_n)\leq e^{-\theta\ell_n/2}.
\end{equation}
\end{lem}
\begin{proof}
Take some intermediate time $t_n=t(\log \ell_n)=o(n)$ where $t>0$ will be chosen later and let $B_n:=[-(1-\varepsilon)\ell_n,\infty)$ with any small $\varepsilon\in(0, 1)$. Observe that as $Z_{t_n}\geq b^{t_n}$,
\begin{align}\label{upperbdweibull1}
\P(M_n\leq m_n-\ell_n)\leq & \P(Z_{t_n}(B_n)\geq b^{t_n}; M_n\leq m_n-\ell_n)+\P(Z_{t_n}(B_n)<b^{t_n})\nonumber\\
\leq & \P(Z_{t_n}(B_n)\geq b^{t_n}; M_n\leq m_n-\ell_n)+\P(Z_{t_n}(B_n^c)\geq 1).
\end{align}
On the one hand, one sees that for $n$ large enough so that $m_n -\varepsilon \ell_n\leq m_{n-t_n}-y^*$,
\begin{align*}
\P(Z_{t_n}(B_n)\geq b^{t_n}; M_n\leq m_n-\ell_n)\leq & \P\l(Z_{t_n}(B_n)\geq b^{t_n}; \max_{|u|=t_n, S_u\in B_n}(S_u+M_{n-t_n}^u)\leq m_n-\ell_n\r)\\
\leq & \P\l(Z_{t_n}(B_n)\geq b^{t_n}; \max_{|u|=t_n, S_u\in B_n}(M_{n-t_n}^u)\leq m_{n}-\varepsilon \ell_n\r)\\
\leq & \P\l(Z_{t_n}(B_n)\geq b^{t_n}; \max_{|u|=t_n, S_u\in B_n}(M_{n-t_n}^u)\leq m_{n-t_n}-y^*\r).
\end{align*}
By Markov property at time $t_n$, all $M^u_{n-t_n}$ are i.i.d. copies of $M_{n-t_n}$ for $|u|=t_n$, and independent of $(S_u, |u|=t_n)$. This yields that
\begin{equation}\label{upperbd1part1}
\P(Z_{t_n}(B_n)\geq b^{t_n}; M_n\leq m_n-\ell_n)\leq  \P\l(M_{n-t_n}\leq m_{n-t_n}-y^*\r)^{b^{t_n}}\leq (1/2)^{b^{t_n}}.
\end{equation}
On the other hand, by Markov property,
\begin{align*}
\P(Z_{t_n}(B_n^c)\geq 1)\leq & \E[Z_{t_n}(B_n^c)]=\E\l[\sum_{|u|=t_n}1_{\{S_u< -(1-\varepsilon)\ell_n\}}\r]\\
=&m^{t_n}\P(S_{t_n}<-(1-\varepsilon)\ell_n)\\
=&m^{t_n}\P(e^{-\theta S_{t_n}}>e^{\theta(1-\varepsilon)\ell_n}),
\end{align*}
where $\theta>0$ such that $\E[e^{-\theta X}]<\infty$.  Again by Markov property, one gets that
\begin{align}\label{upperbd1part2}
\P(Z_{t_n}(B_n^c)\geq 1)\leq &m^{t_n}e^{-\theta(1-\varepsilon)\ell_n}\E[e^{-\theta S_{t_n}}]\nonumber\\
=&m^{t_n}e^{-\theta(1-\varepsilon)\ell_n} \E[e^{-\theta X}]^{t_n}\leq e^{-\theta(1-\varepsilon)\ell_n+\Theta( t_n)}.
\end{align}
Going back to \eqref{upperbdweibull1}, by \eqref{upperbd1part1} and \eqref{upperbd1part2}, one concludes that
\[
\P(M_n\leq m_n-\ell_n)\leq e^{-cb^{t_n}}+ e^{-\theta(1-\varepsilon)\ell_n+c t_n}.
\]
Here we choose $t= 2/\log b$ so that $b^{t_n}=\ell_n^2\gg \theta \ell_n\gg t_n$. Consequently, for arbitrary small $\varepsilon>0$, and for sufficiently large $n$,
\begin{equation*}
 \P(M_n\leq m_n-\ell_n)\leq e^{-\theta(1-\varepsilon)\ell_n}.
\end{equation*}
\end{proof}

\paragraph{The case $\alpha=1$}
This case is relatively simple. Take some intermediate time $t_n=t\log \ell_n$ where $t>0$ will be chosen later. Recall that $B_n=[-(1-\varepsilon)\ell_n,\infty)$ with arbitrary small $\varepsilon\in(0,1)$. Observe that for any $\delta\in(0, 1/b)$,
\begin{align}\label{upbd20}
&\P(M_n\leq m_n-\ell_n)\leq  \P(Z_{t_n}(B_n)\geq \delta b^{t_n}; M_n\leq m_n-\ell_n)+\P(Z_{t_n}(B_n)<\delta b^{t_n})\nonumber\\
\leq & \P\l(Z_{t_n}(B_n)\geq \delta b^{t_n}; \max_{|u|=t_n, S_u\in B_n}(S_u+M^u_{n-t_n})\leq m_n-\ell_n\r)+\P(Z_{t_n}(B_n)<\delta b^{t_n}).
\end{align}
On the one hand, one sees that for $n$ large enough so that $m_n-\ell_n+(1-\varepsilon)\ell_n\leq m_{n-t_n}-\varepsilon\ell_n/2$,
\begin{multline*}
\P\l(Z_{t_n}(B_n)\geq \delta b^{t_n}; \max_{|u|=t_n, S_u\in B_n}(S_u+M^u_{n-t_n})\leq m_n-\ell_n\r)\\
\leq \P\l(Z_{t_n}(B_n)\geq \delta b^{t_n}; \max_{|u|=t_n, S_u\in B_n}(M^u_{n-t_n})\leq m_{n-t_n}-\varepsilon\ell_n/2\r).
\end{multline*}
By Markov property at time $t_n$, all $M^u_{n-t_n}$, $|u|=t_n$ are i.i.d. copies of $M_{n-t_n}$, and independent of $Z_{t_n}(\cdot)$. This yields that
\begin{align}\label{upbd20part1}
\P\l(Z_{t_n}(B_n)\geq \delta b^{t_n}; \max_{|u|=t_n, S_u\in B_n}(M^u_{n-t_n})\leq m_{n-t_n}-\varepsilon\ell_n/2\r)\leq &\P(M_{n-t_n}\leq m_{n-t_n}-\varepsilon\ell_n/2)^{\delta b^{t_n}}\nonumber\\
\leq & e^{-\lambda \varepsilon\ell_n/8 \times \delta b^{t_n}},
\end{align}
where the last inequality follows from \eqref{roughupbd1}.

On the other hand, since $\delta<1/b$, the event $Z_{t_n}(B_n)<\delta b^{t_n}$ implies that for any $|v|=1$, $\{|u|=t_n: u>v\}\not\subset\{|u|=t_n, S_u\in B_n\}$. This means that
\begin{align*}
\P(Z_{t_n}(B_n)<\delta b^{t_n})\leq & \P\l(\cap_{|v|=1}\cup_{|u|=t_n, u>v}\{S_u\in B_n^c\}\r)\nonumber\\
\leq &\E\l[\P\l(\cup_{|u|=t_n, u>v}\{S_u\in B_n^c\}\r)^{Z_1}\r] \nonumber\\
\leq & \E\l[\l(\E\l(\sum_{|u|=t_n, u>v}1_{\{S_u\in B_n^c\}}\Big\vert |v|=1\r)\r)^{b}\r],
\end{align*}
where the last inequality follows from the fact that $Z_1\geq b$ and  Markov inequality. By independence between offsprings and motions, this leads to
\begin{align*}
\P(Z_{t_n}(B_n)<\delta b^{t_n})\leq & \l(\E\l[Z_{t_n-1}\r]\P\{S_{t_n}\in B_n^c\}\r)^{b}\\
\leq & \l(m^{t_n-1}\P\{S_{t_n}\leq -(1-\varepsilon)\ell_n\}\r)^{b}\\
\leq & m^{b(t_n-1)} \l(e^{-\theta(1-\ez) \ell_n} \E[e^{-\theta S_{t_n}}]\r)^b,
\end{align*}
where the last inequality holds by Markov inequality for any $\theta\in(0,\lambda)$. We hence end up with
\begin{equation}\label{upbd20part2}
\P(Z_{t_n}(B_n)<\delta b^{t_n})\leq m^{b(t_n-1)} e^{-\theta b(1-\ez)\ell_n} \E[e^{-\theta X}]^{bt_n}=e^{-\theta b(1-\ez)\ell_n+\Theta(t_n)},
\end{equation}
for any $\theta\in(0,\lambda)$. In view of \eqref{upbd20}, \eqref{upbd20part1} and \eqref{upbd20part2}, one obtains that for any $\varepsilon\in(0, 1)$,
\[
\P(M_n\leq m_n-\ell_n)\leq e^{-\lambda \varepsilon\ell_n/8 \times \delta b^{t_n}}+e^{-\lambda(1-\varepsilon) b\ell_n+\Theta(t_n)}.
\]
For any choice of  $t_n=\Theta(\log \ell_n)$ so that $b^{t_n}\gg 1$, we could conclude that for arbitrary small $\varepsilon>0$,
\[
\limsup_{n\rightarrow\infty}\frac{1}{\ell_n}\log \P(M_n\leq m_n-\ell_n)\leq - \lambda(1-\varepsilon) b.
\]
\paragraph{The case $\alpha>1$}

We are going to use Lemma \eqref{rupbd1}. Note that $\E[e^{-\theta X}]<\infty$ for any $\theta>0$ because $\alpha>1$. It brings out that for all $n$ large enough,
\begin{equation}\label{roughupbd2}
\P(M_n\leq m_n-\ell_n)\leq e^{-2\ell_n}.
\end{equation}

We still use some intermediate time $t_n^+=t^+\log \ell_n$ which will be determined later. The rouge idea is similar to what we used above. Recall that $B_n=[-(1-\varepsilon)\ell_n,\infty)$ with $\varepsilon\in(0,1)$. Observe that for $\delta_n:=\delta \log \ell_n$ with some $\delta\in(0,t^+)$,
\begin{equation}\label{upbd2}
\P(M_n\leq m_n-\ell_n)\leq  \P(Z_{t_n^+}(B_n)\geq b^{t_n^+-\delta_n}; M_n\leq m_n-\ell_n)+\P(Z_{t_n^+}(B_n)< b^{t_n^+-\delta_n}).
\end{equation}
Similarly to \eqref{upperbd1part1}, by Markov property at time $t_n^+$, one has
\begin{align*}
\P(Z_{t_n^+}(B_n)\geq b^{t_n^+-\delta_n}; M_n\leq m_n-\ell_n) \leq &\P\l(M_{n-t_n^+}\leq m_n-\varepsilon\ell_n\r)^{ b^{t_n^+-\delta_n}}\\
\leq & \P\l(M_{n-t_n^+}\leq m_{n-t_n^+}-\varepsilon\ell_n/2\r)^{ b^{t_n^+-\delta_n}}.
\end{align*}
By use of the rough upper bound \eqref{roughupbd2}, we get that
\begin{equation}\label{upbd2part1}
\P(Z_{t_n^+}(B_n)\geq b^{t_n^+-\delta_n}; M_n\leq m_n-\ell_n) \leq e^{-\varepsilon \ell_n b^{t_n^+-\delta_n}}.
\end{equation}
It remains to bound $\P(Z_{t_n^+}(B_n)<b^{t_n^+-\delta_n})$. Let $\mathbf{t}$ denote a fixed tree of $t_n^+$ generations and $\P^\mathbf{t}(\cdot)$ denote the conditional probability $\P(\cdot \vert {\cal T}_{t_n^+}=\mathbf{t})$ where ${\cal T}_{t_n^+}$ denotes the genealogical tree $\cal T$ up to the $t_n^+$-th generation. Observe that
\begin{equation}\label{keyupbd}
\P(Z_{t_n^+}(B_n)<b^{t_n^+-\delta_n})= \sum_{\mathbf{t}}\P( {\cal T}_{t_n^+}=\mathbf{t})\P^\mathbf{t}(Z_{t_n^+}(B_n)\leq b^{t_n^+-\delta_n}).
\end{equation}
Here for convenience, we replace each displacement $X_u$ by $X_u^+:=(-X_u)\vee M$ for some large and fixed constant $M\geq1$. Now denote the new positions achieved by these new displacements by
\[
S_u^+:=\sum_{\rho\prec v\preceq u}X_v^+, \qquad\forall |u|\leq t_n^+.
\]
Obviously, $S_u^+\geq \sum_{\rho\prec v\preceq u}(-X_v)=-S_u$. So, if $Z_{t_n^+}(B_n)\leq b^{t_n^+-\delta_n}$, then
\[
\sum_{|u|=t_n^+} 1_{\{S_u^+\leq (1-\varepsilon)\ell_n\}}\leq \sum_{|u|=t_n^+}1_{\{S_u\in B_n\}}=Z_{t_n^+}(B_n)\leq b^{t_n^+-\delta_n}.
\]
Therefore, for $\varepsilon\in(0,1/2)$ and for $n$ sufficiently large so that $t_n^+=t^+\log \ell_n\leq \varepsilon\ell_n$,
\begin{align*}
\P^\mathbf{t}(Z_{t_n^+}(B_n)\leq b^{t_n^+-\delta_n})\leq &\P^{\mathbf{t}}\l(\sum_{|u|=t_n^+} 1_{\{S_u^+\leq (1-\varepsilon)\ell_n\}}\leq b^{t_n^+-\delta_n}\r)\\
\leq &\sum_{x_u\in\mbb N \cap[M,\infty); u\in\mathbf{t}} \prod_{u\in\mathbf{t}}\P(X_u^+\in [x_u,x_u+1)) 1_{\{\sum_{|u|=t_n^+} 1_{\{s_u\leq (1-2\varepsilon)\ell_n\}}\leq b^{t_n^+-\delta_n}\}},
\end{align*}
where $s_u:=\sum_{\rho\prec v\preceq u}x_v$. We regard $\{x_u, u\in\mathbf{t}\}$ as a marked tree. Here by manipulating the order of $u\in\mathbf{t}$, we could construct a new marked tree $\{x_u, u\in\mathbf{t}_*\}$, where the lexicographical orders of individuals are totally rearranged so that the most recent common ancestor $u^*$ of individuals located below $(1-2\varepsilon)\ell_n$ at the $t_n^+$-th generation is of the generation $s_n$ with $t_n^+\geq s_n\geq \delta_n$. However, $\mathbf{t}_*$ and $\mathbf{t}$, viewed as sets of individuals, contain exactly the same individuals. The detailed construction will be explained later.

Now we cut this $u^*$ and remove all its descendants from $\mathbf{t}_*$ to get a pruned tree $\mathbf{t}_*^{\backslash u^*}$. Note that all individuals of this tree $\mathbf{t}_*^{\backslash u^*}$ up to the generation $t_n^+-1$ have at least $b$ children,  except the parent of $u^*$ . And the parent of $u^*$ has at least $b-1$ children. So we can extract from $\mathbf{t}_*^{\backslash u^*}$ an "almost" $b$-ary regular tree $\mathbf{t}_b^{\backslash u^*}$ so that its all descendants are located above $(1-2\varepsilon)\ell_n$. Here in $\mathbf{t}_b^{\backslash u^*}$, the parent of $u^*$ has $b-1$ children, and all others except the leaves have exactly $b$ children.

This operation leads to the following estimation, for any fixed tree $\mathbf{t}$ such that $\P( {\cal T}_{t_n^+}=\mathbf{t})>0$,
\begin{align}\label{upbd2part2}
\P^\mathbf{t}(Z_{t_n^+}(B_n)\leq b^{t_n^+-\delta_n})\leq&  \sum_{x_u\in\mbb N \cap[M,\infty);u\in\mathbf{t}_*}\prod_{u\notin\mathbf{t}_b^{\backslash u^*}}\P(X_u^+\in [x_u,x_u+1)) \nonumber\\
&\qquad\qquad\times\prod_{u\in\mathbf{t}_b^{\backslash u^*}}\P(X_u^+\in [x_u,x_u+1)) 1_{\{s_u\geq (1-2\varepsilon)\ell_n; \forall u\in \mathbf{t}_b^{\backslash u^*} s.t. |u|=t_n^+\}}\nonumber\\
\leq &\sum_{x_u\in\mbb N \cap[M,\infty); u\in\mathbf{t}_b^{\backslash u^*}} \prod_{u\in\mathbf{t}_b^{\backslash u^*}}\P(X_u^+\in [x_u,x_u+1)) 1_{\{s_u\geq (1-2\varepsilon)\ell_n; \forall u\in \mathbf{t}_b^{\backslash u^*} s.t. |u|=t_n^+\}}\nonumber\\
\leq &\Sigma_{\mathbf{t}_b^{\backslash u^*}, A}+ \P^{\mathbf{t}_b^{\backslash u^*}}(\exists u\in \mathbf{t}_b^{\backslash u^*}\textrm{ such that } X_u^+\geq A\ell_n),
\end{align}
where
\[
\Sigma_{\mathbf{t}_b^{\backslash u^*}, A}:=\sum_{x_u\in\mbb N \cap[M,A\ell_n); u\in\mathbf{t}_b^{\backslash u^*}} \prod_{u\in\mathbf{t}_b^{\backslash u^*}}\P(X_u^+\in [x_u,x_u+1)) 1_{\{s_u\geq (1-2\varepsilon)\ell_n; \forall u\in \mathbf{t}_b^{\backslash u^*} s.t. |u|=t_n^+\}}.
\]
 As the total progeny of $\mathbf{t}_b^{\backslash u^*}$ less than $ \sum_{k=1}^{t_n^+}b^k$,
\begin{align}\label{upbd2part2small}
\P^{\mathbf{t}_b^{\backslash u^*}}(\exists u\in \mathbf{t}_b^{\backslash u^*}\textrm{ such that } X_u^+\geq A\ell_n)\leq &\l(\sum_{k=1}^{t_n^+}b^k\r)\P(X^+\geq A\ell_n)\nonumber\\
\leq & C b^{t_n^++1}e^{-\lambda (A\ell_n)^\alpha},
\end{align}
where the last inequality follows from \eqref{supexptail}.
On the other hand, observe that
\begin{align*}
&\Sigma_{\mathbf{t}_b^{\backslash u^*}, A}\cr&\leq  (CA\ell_n)^{b^{t_n^++1}} \max_{\forall u\in\mathbf{t}_b^{\backslash u^*};x_u\in\mbb N \cap[M,A\ell_n);}\exp\l\{-\lambda\sum_{u\in\mathbf{t}_b^{\backslash u^*}}x_u^\alpha\r\}1_{\{s_u\geq (1-2\varepsilon)\ell_n; \forall u\in \mathbf{t}_b^{\backslash u^*} s.t. |u|=t_n^+\}}.
\end{align*}
Here we claim that
\begin{multline}\label{upbdSigma1}
\max_{x_u\in\mbb N \cap[M,A\ell_n); \forall u\in\mathbf{t}_b^{\backslash u^*}}\exp\l\{-\lambda\sum_{u\in\mathbf{t}_b^{\backslash u^*}}x_u^\alpha\r\}1_{\{s_u\geq (1-2\varepsilon)\ell_n; \forall u\in \mathbf{t}_b^{\backslash u^*} s.t. |u|=t_n^+\}}\\
\leq \exp\l\{-\lambda (b_\alpha-1)^{\alpha-1} (1-b^{-s_n})^{\alpha+1} (1-2\varepsilon)^\alpha \ell_n^\alpha\r\},
\end{multline}
with $b_\alpha=b^{\frac{1}{\alpha-1}}$.
The proof of \eqref{upbdSigma1} will be postponed to the end of this section. Let us admit it now so that
\begin{equation}\label{upbd2part2main}
\Sigma_{\mathbf{t}_b^{\backslash u^*}, A}\leq  (CA\ell_n)^{b^{t_n^++1}} \exp\l\{-\lambda (b_\alpha-1)^{\alpha-1} (1-b^{-s_n})^{\alpha+1} (1-2\varepsilon)^\alpha \ell_n^\alpha\r\}.
\end{equation}
Plugging \eqref{upbd2part2small} and \eqref{upbd2part2main} into \eqref{upbd2part2} yields that
\begin{multline*}
\P^\mathbf{t}(Z_{t_n^+}(B_n)\leq b^{t_n^+-\delta_n})\cr
\leq C b^{t_n^++1}e^{-\lambda (A\ell_n)^\alpha}+(CA\ell_n)^{b^{t_n^++1}} \exp\l\{-\lambda (b_\alpha-1)^{\alpha-1} (1-b^{-s_n})^{\alpha+1} (1-2\varepsilon)^\alpha \ell_n^{\alpha}\r\}.
\end{multline*}
Plugging it into \eqref{keyupbd} brings out that
\begin{multline}\label{keybd}
\P(Z_{t_n^+}(B_n)\leq b^{t_n^+-\delta_n})\\
\leq C b^{t_n^++1}e^{-\lambda (A\ell_n)^\alpha}+(CA\ell_n)^{b^{t_n^++1}} \exp\l\{-\lambda (b_\alpha-1)^{\alpha-1} (1-b^{-s_n})^{\alpha+1} (1-2\varepsilon)^\alpha \ell_n^{\alpha}\r\}.
\end{multline}
\eqref{keybd}, combined with \eqref{upbd2} and \eqref{upbd2part1}, implies that
\begin{multline*}
\P(M_n\leq m_n-\ell_n) \\
\leq e^{-\varepsilon \ell_n b^{t_n^+-\delta_n}} + C e^{-\lambda (A\ell_n)^\alpha+\Theta(\log \ell_n)}+e^{-\lambda (b_\alpha-1)^{\alpha-1} (1-b^{-s_n})^{\alpha+1} (1-2\varepsilon)^\alpha \ell_n^{\alpha}+ \Theta (b^{t_n^+}\log (A \ell_n))},
\end{multline*}
with $t_n^+=t^+\log \ell_n$, $\delta_n=\delta \log \ell_n$ and $s_n\geq \delta_n$. We choose here a large and fixed $A\geq1$, $t^+=\frac{3\alpha-1}{3\log b}$ and $\delta=\frac{1}{3\log b}$ so that
\[
\ell_n b^{t_n^+-\delta_n}=\ell_n^{\alpha+1/3},\quad A^\alpha \geq 2 (b_\alpha-1)^{\alpha-1} \textrm{ and }  b^{t_n^+}\log (A \ell_n)=o(\ell_n^{\alpha}).
\]
Consequently, letting $n\uparrow\infty$ and then $\varepsilon\downarrow\infty$ shows that
\[
\limsup_{n\rightarrow\infty}\frac{1}{\ell_n^\alpha}\log \P(M_n\leq m_n-\ell_n)\leq -\lambda (b_\alpha-1)^{\alpha-1},
\]
which is what we need.

To complete our proof, let us explain the construction of $\mathbf{t}_b^{\backslash u^*}$ here.

\paragraph{Construction of $\mathbf{t}_b^{\backslash u^*}$.}

For a deterministic sample of the branching random walk up to the generation $t_n^+$, saying $\{s_u: u\in {\bf t}\}$, we construct ${\bf t}^*$ and ${\bf t}_b^{\backslash u^*}$ in the following way. Let $\ell:=(1-2\varepsilon)\ell_n$ and we shall colour the individuals in the backwardly.

At the $t_n^+$-th generation, there are at most $b^{t_n^+-\delta_n}$ individuals positioned below $\ell$, which are all coloured blue.  The other individuals above $\ell$ are coloured red.

At the $(t_n^+-1)$-th generation, the individuals are called $u_{(1)},\, u_{(2)},\cdots,\, u_{(|{\bf t}|_{t_n^+-1})}$ according to their positions such that $s_{u_{(1)}}\geq s_{u_{(2)}}\geq \cdots\geq s_{u_{(|{\bf t}|_{t_n^+-1})}},$ where $|{\bf t}|_{t_n^+-1}=\#\{u\in {\bf t}: |u|=t_n^+-1\}$. Let us start with $u_{(1)}$ and its children. If all children of $u_{(1)}$ are red, then we turn to $u_{(2)}$. Otherwise, we keep its red children and replace its blue children by the red children of other individuals of the  $(t_n^+-1)$-th generation. More precisely, saying that there are $b_1$ blue children of $u_{(1)}$, we collect the red children of $u_{(2)}$ and then the red children of $u_{(3)},\, \cdots$, until we find exactly $b_1$ red ones to be exchanged with the original $b_1$ blue children of $u_{(1)}$.

 \begin{figure}

\includegraphics[width=15cm]  {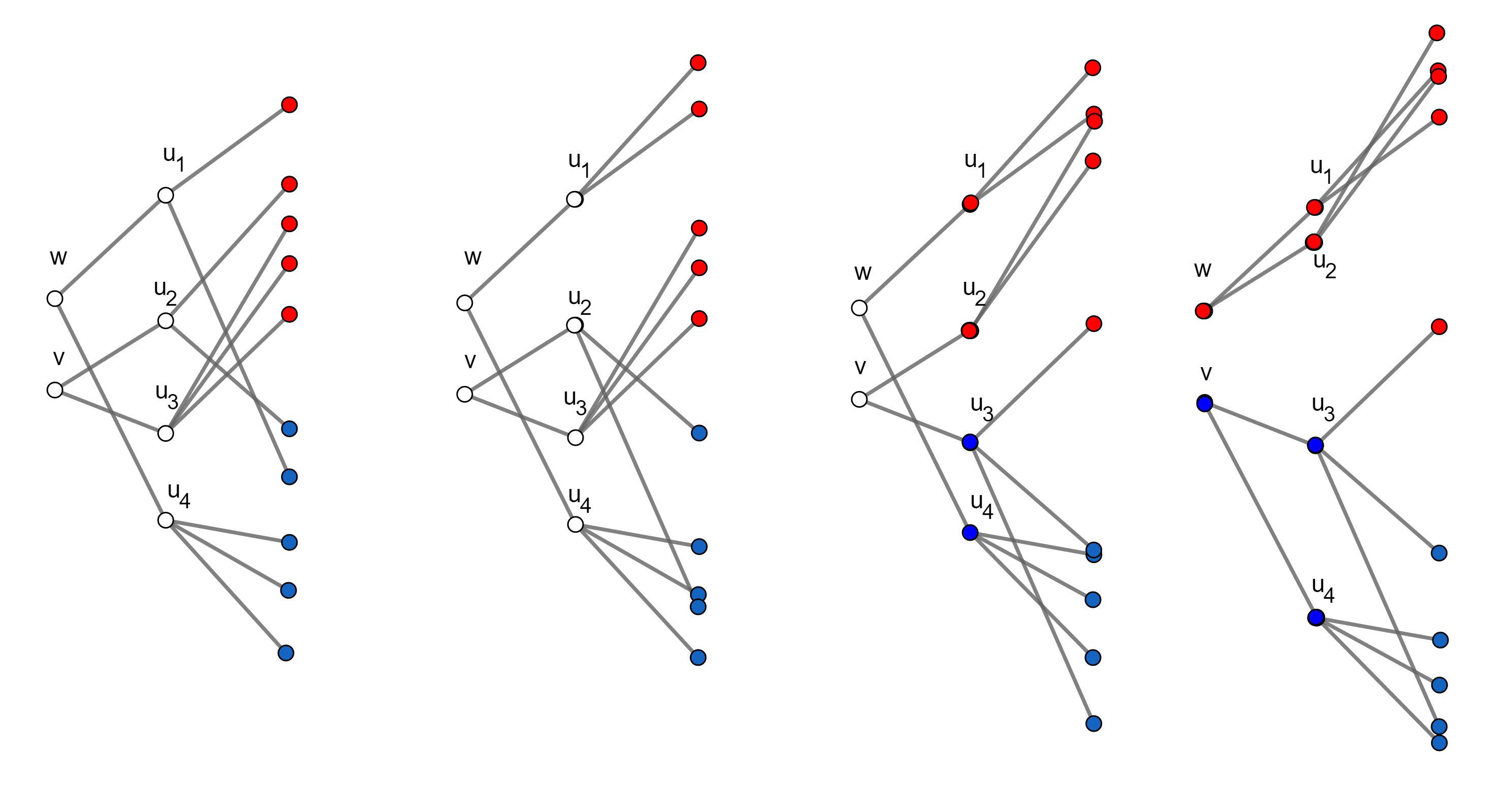}

 \caption{We first exchange $u_1$'s blue child with $u_2$'s red child; then we exchange $u_2$'s blue children with two of $u_3$'s red children. So we color $u_1$ and $u_2$  red and color $u_3$ and $u_4$ blue (Notice that one of $u_3$'s children is red.) Next, we exchange $u_2$ and its subtree with $u_4$ and its subtree. Then $w$ is colored red and $v$ is colored blue. }
 \end{figure}

 When we exchange two individuals $w$ and $v$, we exchange two subtrees rooted at $w$ and $v$, as well as their displacements; see Figure 2. Therefore, the positions of red individuals get higher, and obviously stay above $\ell$.

\begin{figure}
 \center{\includegraphics[width=10cm]  {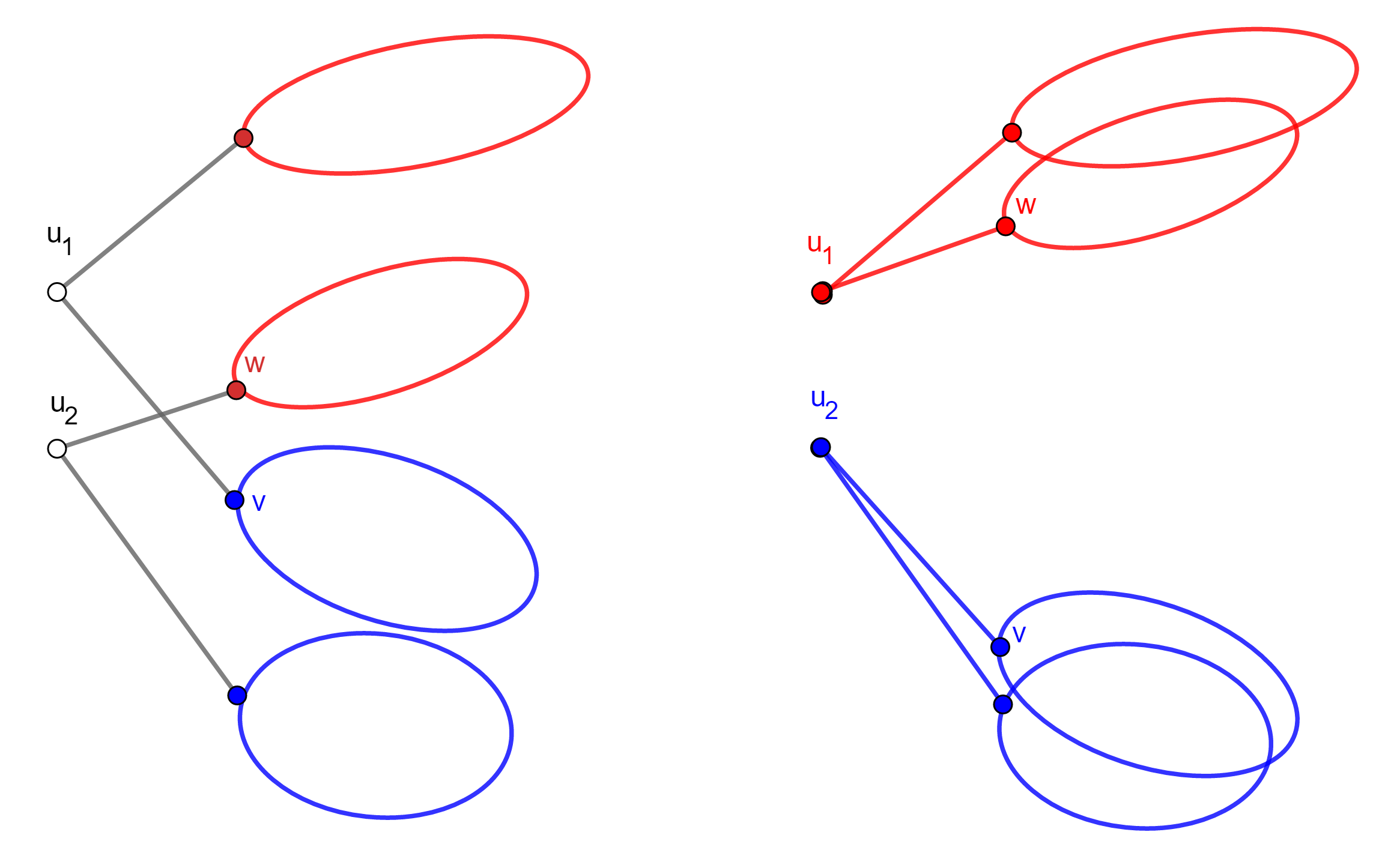}}
 \caption{ Both $u_1$ and $u_2$ have two offsprings. After exchanging subtrees rooted at $w$ and $v$, $u_1$ is colored red and  $u_2$ is colored blue.}
 \end{figure}

Note that in this way the number of children $u_{(1)}$ is unchanged and that all of them are positioned above $\ell$ and red.  Now, we put $u_{(1)}$ aside and restart from $u_{(2)}$ by doing the same exchanges with $u_{(3)}, u_{(4)}, \cdots$. We would stop at some $u_{(k)}$ such that there is no red child left for $u_{(k+1)}, \cdots.$ At this stage, there are at most 3 types of individuals at the $(t_n^+-1)$-th generation: the ones with only red children; the ones with only blue children and the one with red children and blue children (Note that there is at most one individual who has both red and blue children).
Then the individuals with only red children are all coloured red. The others of the $(t_n^+-1)$-th generation are coloured blue. Notice that the number of blue individuals of the $(t_n^+-1)$-th generation are at most $ b^{t_n^+-\delta_n-1}$.

By iteration, we exchange individuals and colour the tree from one generation to the previous generation. Finally, we stop at some generation $s_n$ where only one individual is coloured blue for the first time. We hence obtain the new tree $\mathbf{t}^*$ and find that the ancestor $u^*$ of blue ones is of generation $s_n\geq \delta_n$. Observe that, for all red individuals,  their descendants at $t_n^+$-th generation are positioned above $\ell$. 

\paragraph{Proof of \eqref{upbdSigma1}.}

We shall find a suitable lower bound for $\sum_{u\in\mathbf{t}_b^{\backslash u^*}}x_u^\alpha$ given that $(x_u; u\in\mathbf{t}_b^{\backslash u^*})\in\{s_u\geq (1-2\varepsilon)\ell_n; \forall u\in \mathbf{t}_b^{\backslash u^*} s.t. |u|=t_n^+\}\cap\{x_u\in\mbb N \cap[M,A\ell_n); \forall u\in\mathbf{t}_b^{\backslash u^*}\}$. Let us first consider the restrictions for $x_u$. Note that if $(x_u; u\in\mathbf{t}_b^{\backslash u^*})\in\{s_u\geq (1-2\varepsilon)\ell_n; \forall u\in \mathbf{t}_b^{\backslash u^*} s.t. |u|=t_n^+\}$, then
\begin{equation}\label{cond101}
\sum_{|u|=t_n^+, u\in\mathbf{t}_b^{\backslash u^*}}s_u\geq |\mathbf{t}_b^{\backslash u^*}|_{t_n^+} (1-2\varepsilon)\ell_n,
\end{equation}
where $|\mathbf{t}_b^{\backslash u^*}|_k:=\sum_{|u|=t_n^+}1_{\{u\in \mathbf{t}_b^{\backslash u^*}\}}$ denotes the population of the $k$-th generation of $\mathbf{t}_b^{\backslash u^*}$. We further observe that
\[
\sum_{|u|=t_n^+, u\in \mathbf{t}_b^{\backslash u^*}}s_u=\sum_{|u|=t_n^+}\sum_{\rho\prec v\preceq u}x_v=\sum_{k=1}^{t_n^+} \sum_{|v|=k}\l(x_v\sum_{|u|=t_n^+}1_{\{v\preceq u\}}\r)
\]
where $\sum_{|u|=t_n^+}1_{\{v\preceq u\}}\leq b^{t_n^+-|v|}$ as $\mathbf{t}_b^{\backslash u^*}$ is a pruned $b$-ary tree. Therefore, \eqref{cond101} implies
\begin{equation}\label{restrictions}
\sum_{k=1}^{t_n^+} \sum_{|v|=k} x_v b^{t_n^+-k}\geq |\mathbf{t}_b^{\backslash u^*}|_{t_n^+}  (1-2\varepsilon)\ell_n.
\end{equation}
Recall that the generation of $u^*$ is $s_n$. So,
\begin{equation}\label{popbary}
|\mathbf{t}_b^{\backslash u^*}|_k= b^k, \forall 1\leq k\leq s_n-1;\textrm{ and } |\mathbf{t}_b^{\backslash u^*}|_k= b^{k}-b^{k-s_n}, \forall s_n\leq k\leq t_n^+.
\end{equation}
Let $\overline{x}_k:=\frac{\sum_{|v|=k, v\in \mathbf{t}_b^{\backslash u^*}}x_v}{|\mathbf{t}_b^{\backslash u^*}|_k}$ be the averaged displacement at the $k$-th generation. Then,
\[
\sum_{k=1}^{t_n^+} \sum_{|v|=k} x_v b^{t_n^+-k}= b^{t_n^+}\l(\sum_{k=1}^{s_n-1} \overline{x}_k + \frac{b^k-b^{k-s_n}}{b^k}\sum_{k=s_n}^{t_n^+} \overline{x}_k\r)\leq  b^{t_n^+}\sum_{k=1}^{t_n^+} \overline{x}_k.
\]
Thus, if \eqref{restrictions} holds, one has
\begin{equation}\label{cond102}
\sum_{k=1}^{t_n^+} \overline{x}_k\geq \frac{|\mathbf{t}_b^{\backslash u^*}|_{t_n^+}}{b^{t_n^+}}  (1-2\varepsilon)\ell_n=(1-b^{-s_n})(1-2\varepsilon)\ell_n.
\end{equation}
Hence, \eqref{cond101} implies \eqref{cond102}. This means that
\begin{equation}\label{cond}
\{s_u\geq (1-2\varepsilon)\ell_n; \forall u\in \mathbf{t}_b^{\backslash u^*} s.t. |u|=t_n^+\}\subset\l\{\sum_{k=1}^{t_n^+} \overline{x}_k\geq (1-b^{-s_n})(1-2\varepsilon)\ell_n\r\}.
\end{equation}
So it suffices to find a suitable lower bound of $\sum_{u\in\mathbf{t}_b^{\backslash u^*}}x_u^\alpha$ under the condition that
\begin{equation}\label{cond1}
\sum_{k=1}^{t_n^+} \overline{x}_k\geq (1-b^{-s_n})(1-2\varepsilon)\ell_n.
\end{equation}
 In fact, by convexity on $\mbb R_+$ of $x\mapsto x^\alpha$ for $\alpha>1$,
\[
\sum_{u\in\mathbf{t}_b^{\backslash u^*}}x_u^\alpha=\sum_{k=1}^{t_n^+} \sum_{|u|=t_n^+, u\in \mathbf{t}_b^{\backslash u^*}}x_u^\alpha\geq \sum_{k=1}^{t_n^+} |\mathbf{t}_b^{\backslash u^*}|_k (\overline{x}_k)^\alpha.
\]
Immediately it follows from \eqref{popbary} that
\begin{equation}\label{bdnotree}
\sum_{u\in\mathbf{t}_b^{\backslash u^*}}x_u^\alpha\geq (1-b^{-s_n})\sum_{k=1}^{t_n^+} b^k (\overline{x}_k)^\alpha.
\end{equation}
Let us take a positive sequence $(\mu_k)_{k\geq1}$, which will be determined later, with $\mu_\alpha:=\sum_{k=1}^{t_n^+}\mu_k^\alpha$ and write
\[
\sum_{k=1}^{t_n^+} b^k (\overline{x}_k)^\alpha =\mu_\alpha\sum_{k=1}^{t_n^+}\frac{\mu_k^\alpha}{\mu_\alpha} (\mu_k^{-1}b^{k/\alpha}\overline{x}_k)^\alpha
\]
which again by convexity implies that
\[
\sum_{k=1}^{t_n^+} b^k (\overline{x}_k)^\alpha \geq \mu_\alpha \l(\sum_{k=1}^{t_n^+}\frac{\mu_k^\alpha}{\mu_\alpha} \mu_k^{-1}b^{k/\alpha}\overline{x}_k\r)^\alpha=\mu_\alpha^{1-\alpha}\l(\sum_{k=1}^{t_n^+}\mu_k^{\alpha-1}b^{k/\alpha}\overline{x}_k\r)^\alpha.
\]
We choose $\mu_k=b^{-\frac{k}{\alpha(\alpha-1)}}$ so that $\mu_k^{\alpha-1}b^{k/\alpha}\overline{x}_k=\overline{x}_k$ for any $k\geq1$. Thus,
\[
\mu_\alpha=\sum_{k=1}^{t_n^+} b^{-\frac{k}{\alpha-1}}\leq\frac{1}{b^{\frac{1}{\alpha-1}}-1}
\]
and
\begin{equation}\label{cond2}
\sum_{k=1}^{t_n^+} b^k (\overline{x}_k)^\alpha \geq \l(\frac{1}{b^{\frac{1}{\alpha-1}}-1}\r)^{1-\alpha}\l(\sum_{k=1}^{t_n^+} \overline{x}_k\r)^\alpha \geq (b_\alpha-1)^{\alpha-1} (1-b^{-s_n})^\alpha (1-2\varepsilon)^\alpha \ell_n^{\alpha},
\end{equation}
where the last inequality follows from \eqref{cond1}. Plugging \eqref{cond2} into \eqref{bdnotree} shows that
\[
\sum_{u\in\mathbf{t}_b^{\backslash u^*}}x_u^\alpha\geq (1-b^{-s_n})(b_\alpha-1)^{\alpha-1} (1-b^{-s_n})^\alpha (1-2\varepsilon)^\alpha \ell_n^{\alpha}.
\]
This suffices to conclude \eqref{upbdSigma1}.

\subsection{Proof of Theorem \ref{gumbeltail}: step size of Gumbel tail}

The arguments for Gumbel tail are similar to that for Weibull tail.

\subsubsection{Lower bound of Theorem \ref{gumbeltail}}

We are going to demonstrate that
\[
\P(M_n\leq m_n-\ell_n)\geq \exp\{-e^{\beta(\alpha,b)\ell_n^{\frac{\alpha}{\alpha+1}}+o(\ell_n^{\frac{\alpha}{\alpha+1}})}\},
\]
where $\beta(\alpha,b):=\left(\frac{1+\alpha}{\alpha} \log b\right)^{\frac{\alpha}{\alpha+1}}$.

By the assumption of Theorem \ref{gumbeltail}, there exist two constants $0<c<C<\infty$ such that for any $x\geq 0$,
\begin{equation}\label{doublexptail}
ce^{-e^{ x^\alpha}}\leq \P(X<-x) \leq C e^{-e^{ x^\alpha}}.
\end{equation}
Note that here $\alpha>0$. Using the similar arguments as in Section \ref{lowerbdweibull}, we take some intermediate time $t_n^-=o(\ell_n)$ and a positive sequence $(a_k)_{1\leq k\leq t_n^-}$. Then, observe that
\begin{align*}
&\P(M\leq m_n-\ell_n)\geq  \P\l(Z_{t_n^-}=b^{t_n^-}; \forall |u|=k\in\{1,\cdots,t_n^-\}, X_u\leq -a_k; M_n\leq m_n-\ell_n \r)\\
\geq &  \P\l(Z_{t_n^-}=b^{t_n^-}; \forall |u|=k\in\{1,\cdots,t_n^-\}, X_u\leq -a_k\r)\P\l( M_{n-t_n^-}\leq m_n-\ell_n+\sum_{k=1}^{t_n^-}a_k \r)^{b^{t_n^-}}\\
=& p_b^{\sum_{k=0}^{t_n^--1}b^k}\prod_{k=1}^{t_n^-} \P(X<-a_k)^{b^k}\P\l( M_{n-t_n^-}\leq m_n-\ell_n+\sum_{k=1}^{t_n^-}a_k \r)^{b^{t_n^-}}.
\end{align*}
By \eqref{doublexptail}, one has
\begin{equation}\label{lowerbdgumbel}
\P(M\leq m_n-\ell_n)\geq p_b^{\frac{b^{t_n^-}-1}{b-1}} c^{\frac{b^{t_n^-+1}-b}{b-1}}\exp\l\{-\sum_{k=1}^{t_n^-}e^{a_k^\alpha}b^k\r\} \P\l( M_{n-t_n^-}\leq m_n-\ell_n+\sum_{k=1}^{t_n^-}a_k \r)^{b^{t_n^-}}.
\end{equation}
Here we take $t_n^-:=t^- \ell_n^{\frac{\alpha}{\alpha+1}}$ and $a_k:=(\log b)^{1/\alpha}( t_n^-+1-k)^{1/\alpha}$ with $t^-:=\l(\frac{1+\alpha}{\alpha}a\r)^{\frac{\alpha}{\alpha+1}}(\log b)^{-\frac{1}{\alpha+1}}$. Now observe that for arbitrary small $\varepsilon>0$ and $n$ large enough,
\begin{align*}
\sum_{k=1}^{t_n^-}a_k =&(\log b)^{1/\alpha} \sum_{k=1}^{t_n^-} ( t_n^-+1-k)^{1/\alpha}\geq   (\log b)^{1/\alpha} \int_1^{t_n^-} (t_n^-+1-s)^{1/\alpha}ds \\
= & \ell_n-\Theta(1)\geq \ell_n-(m_n-m_{n-t_n^-}-y^*).
\end{align*}
This leads to the fact that
\begin{align*}
 \P\l( M_{n-t_n^-}\leq m_n-\ell_n+\sum_{k=1}^{t_n^-}a_k \r)^{b^{t_n^-}}\geq & \P\l(M_{n-t_n^-}\leq m_{n-t_n^-}+y^*\r)^{b^{t_n^-}}\\
\geq & e^{-\Theta(b^{t_n^-})}.
\end{align*}
On the other hand, note that
\[
\sum_{k=1}^{t_n^-}e^{a_k^\alpha}b^k=\sum_{k=1}^{t_n^-}b^{t_n^-+1}=bt_n^-e^{(t\log b)\ell_n^{\frac{\alpha}{\alpha+1}}}.
\]
Going back to \eqref{lowerbdgumbel}, as $b^{t_n^-}\ll t_n^-e^{(t\log b)\ell_n^{\frac{\alpha}{\alpha+1}}}$ and $t_n^-=e^{o(\ell_n^{\frac{\alpha}{\alpha+1}})}$, one concludes that
\[
\P(M\leq m_n-\ell_n)\geq \exp\l\{- t_n^-e^{(t\log b)\ell_n^{\frac{\alpha}{\alpha+1}}}- \Theta(b^{t_n^-})\r\}=\exp\{-e^{\beta(\alpha,b)\ell_n^{\frac{\alpha}{\alpha+1}}+o(\ell_n^{\frac{\alpha}{\alpha+1}})}\},
\]
where $\beta(\alpha,b)=t\log b=\l(\frac{1+\alpha}{\alpha}\log b\r)^{\frac{\alpha}{\alpha+1}}$.
\subsubsection{Upper bound of Theorem \ref{gumbeltail}}

We first prove a rough upper bound.

\begin{lem}
Assume \eqref{momoffspring}, \eqref{expmom}, \eqref{hyp3}, \eqref{hyp4} and $b\geq2$. If $\P(X<-x)=\Theta(1)e^{-e^{x^\alpha}}$ as $x\rightarrow\infty$, then there exists $\eta_0>0$ such that for all $n$ large enough,
\begin{equation}\label{roughupbd3}
\P(M_n\leq m_n-\ell_n)\leq \exp(-e^{\eta_0 \ell_n^{\frac{\alpha}{\alpha+1}} }).
\end{equation}
\end{lem}
\begin{proof}
Let $t_n:=t\ell_n^{\frac{\alpha}{\alpha+1}}$
with some $0<t<\infty$. Again, we use $B_n=[-(1-\varepsilon)\ell_n,\infty)$ with $\varepsilon\in(0, 1)$ and observe that by Markov property at time $t_n$,
\begin{align*}
&\P(M_n\leq m_n-\ell_n)\leq \P(Z_{t_n}(B_n)\geq b^{t_n}; M_n\leq m_n-\ell_n)+\P(Z_{t_n}(B_n^c)\geq 1)\\
\leq &\P\l(Z_{t_n}(B_n)\geq b^{t_n}; \max_{|u|=t_n, S_u\in B_n}(M_{n-t_n}^u)\leq m_{n}-\varepsilon \ell_n\r)+\P(Z_{t_n}(B_n^c)\geq 1)\\
\leq &\P\l(M_{n-t_n}\leq m_{n-t_n}-y^*\r)^{b^{t_n}}+\P(Z_{t_n}(B_n^c)\geq 1),
\end{align*}
for all sufficiently large $n$. Again, using $\P\l(M_{n-t_n}\leq m_{n-t_n}-y^*\r)\leq 1/2$ and Markov inequality, one has
\begin{align*}
\P(M_n\leq m_n-\ell_n)\leq& e^{-cb^{t_n}}+\P(Z_{t_n}(B_n^c)\geq 1)\\
\leq & e^{-cb^{t_n}}+ \E\l[\sum_{|u|=t_n}1_{S_u < -(1-\varepsilon)\ell_n}\r]\\
=&e^{-cb^{t_n}}+m^{t_n}\P(S_{t_n}<-(1-\varepsilon)\ell_n).
\end{align*}
Observe that $\{S_{t_n}<-(1-\varepsilon)\ell_n\}$ implies that at least one increment is less than $-(1-\varepsilon)\ell_n/t_n$. Therefore,
\begin{align*}
\P(M_n\leq m_n-\ell_n)\leq& e^{-cb^{t_n}}+m^{t_n} t_n \P(X\leq -(1-\varepsilon)\ell_n/t_n)\\
\leq & \exp(-ce^{t\log b \ell_n^{\frac{\alpha}{\alpha+1}} })+C t_n \exp(-e^{(\frac{1-\varepsilon}{t})^{\alpha} \ell_n^{\frac{\alpha}{\alpha+1}}}+t_n\log m),
\end{align*}
where we choose a small positive $t$ such that $t_n \exp(-e^{(\frac{1-\varepsilon}{t})^{\alpha} \ell_n^{\frac{\alpha}{\alpha+1}}}+t_n\log m)\ll \exp(-e^{t\log b \ell_n^{\frac{\alpha}{\alpha+1}} })$. As a result,  there exists $\eta_0>0$ such that for all $n$ large enough,
\begin{equation*}
\P(M_n\leq m_n-\ell_n)\leq \exp(-e^{\eta_0 \ell_n^{\frac{\alpha}{\alpha+1}} }).
\end{equation*}
\end{proof}
Now we are ready to prove the upper bound.
Let $t_n^+:=t^+\ell_n^{\frac{\alpha}{\alpha+1}}=o(\ell_n)$ and $\delta_n:=\delta \ell_n^{\frac{\alpha}{\alpha+1}}$
with some $0<\delta<t^+<\infty$. Using the similar arguments as in the Subsection \ref{upbdWeibull}, in view of  \eqref{upbd2} and to \eqref{upbd2part1}, one sees that for any $\varepsilon\in(0,1/2)$,
\begin{align*}
\P(M_n\leq m_n-\ell_n)\leq & \P(Z_{t_n^+}(B_n)\geq b^{t_n^+-\delta_n}; M_n\leq m_n-\ell_n)+\P(Z_{t_n^+}(B_n)< b^{t_n^+-\delta_n})\\
\leq & \P(M_{n-t_n^+}\leq m_{n-t_n^+}-\varepsilon \ell_n/2)^{b^{t_n^+-\delta_n}}+\P(Z_{t_n^+}(B_n)<b^{t_n^+-\delta_n}),
\end{align*}
which by \eqref{roughupbd3} is bounded by
\[
\exp\l(-e^{\eta_0 (\varepsilon\ell_n/2)^{\frac{\alpha}{\alpha+1}}}b^{t_n^+-\delta_n}\r)+\P(Z_{t_n^+}(B_n)<b^{t_n^+-\delta_n}).
\]
Similarly to \eqref{upbd2part2}, one also sees that
\begin{align}\label{upbd3}
\P(M_n\leq m_n-\ell_n)\leq & \exp\l(-e^{\eta_0(\varepsilon\ell_n/2)^{\frac{\alpha}{\alpha+1}}}b^{t_n^+-\delta_n}\r)+\P(Z_{t_n^+}(B_n)<b^{t_n^+-\delta_n})\nonumber\\
\leq &\exp\l(-e^{\eta_0 (n-t_n^+)^{\frac{\alpha\beta}{\alpha+1}}}b^{t_n^+-\delta_n}\r)+\Sigma_{\mathbf{t}_b^{\backslash u^*}, A}+ \P^{\mathbf{t}_b^{\backslash u^*}}(\exists |u|\leq t_n^+, X_u\geq A\ell_n),
\end{align}
where $\mathbf{t}_b^{\backslash u^*}$ is a $b$-ary regular tree pruned at some $u^*$ of generation $s_n\geq \delta_n$ and
\[
\Sigma_{\mathbf{t}_b^{\backslash u^*}, A}:=\sum_{x_u\in\mbb N \cap[M,A\ell_n); u\in\mathbf{t}_b^{\backslash u^*}} \prod_{u\in\mathbf{t}_b^{\backslash u^*}}\P(X_u^+\in [x_u,x_u+1)) 1_{\{s_u\geq (1-2\varepsilon)\ell_n; \forall u\in \mathbf{t}_b^{\backslash u^*} s.t. |u|=t_n^+\}}.
\]
On the one hand, by Markov inequality like \eqref{upbd2part2small}, for $A\geq1$ and $n$ sufficiently large,
\begin{align}\label{upbd3smallpart}
\P^{\mathbf{t}_b^{\backslash u^*}}(\exists |u|\leq t_n^+, X_u\geq A\ell_n)\leq &\sum_{k=1}^{t_n^+} \P(X\geq A\ell_n)\nonumber\\
\leq & C b^{t_n^+} e^{-e^{(A\ell_n)^\alpha}}=o_n(1)\P(M_n\leq m_n-\ell_n),
\end{align}
according to the lower bound obtained above. It remains to bound $\Sigma_{\mathbf{t}_b^{\backslash u^*}, A}$. In fact,
\[
\Sigma_{\mathbf{t}_b^{\backslash u^*}, A} \leq (CA\ell_n)^{b^{t_n^++1}} \max_{x_u\in\mbb N \cap[M,A\ell_n); \forall u\in\mathbf{t}_b^{\backslash u^*}}\exp\l\{-\sum_{u\in\mathbf{t}_b^{\backslash u^*}}e^{x_u^\alpha}\r\}1_{\{s_u\geq (1-2\varepsilon)\ell_n; \forall u\in \mathbf{t}_b^{\backslash u^*} s.t. |u|=t_n^+\}},
\]
where we need to bound from below
\begin{equation}\label{upbdSigma2}
\min\l\{\sum_{u\in\mathbf{t}_b^{\backslash u^*}}e^{x_u^\alpha}\Big\vert x_u\in\mbb N \cap[M,A\ell_n); \forall u\in\mathbf{t}_b^{\backslash u^*}; s_u\geq (1-2\varepsilon)\ell_n; \forall u\in \mathbf{t}_b^{\backslash u^*} s.t. |u|=t_n^+\r\}.
\end{equation}
Note that for any $\alpha>0$, there exists $M\geq1$ such that $x\mapsto e^{x^\alpha}$ is convex on $[M,\infty)$. Let us take such $M$ and observe that
\[
\sum_{u\in\mathbf{t}_b^{\backslash u^*}}e^{x_u^\alpha}=\sum_{k=1}^{t_n^+}|\mathbf{t}_b^{\backslash u^*}|_k \sum_{|u|=k}\frac{1}{|\mathbf{t}_b^{\backslash u^*}|_k}e^{x_u^\alpha}\geq \sum_{k=1}^{t_n^+}|\mathbf{t}_b^{\backslash u^*}|_ke^{ \overline{x}_k^\alpha},
\]
where $\overline{x}_k$ denotes the averaged displacements of the $k$-th generation. As $|\mathbf{t}_b^{\backslash u^*}|_k\geq (1-b^{-s_n}) b^{k}$ for any $1\leq k\leq t_n^+$, one gets that
\begin{equation}\label{bdbymax}
\sum_{u\in\mathbf{t}_b^{\backslash u^*}}e^{x_u^\alpha}\geq (1-b^{-s_n})\sum_{k=1}^{t_n^+} b^k e^{ \overline{x}_k^\alpha}\geq (1-b^{-s_n}) e^{\Xi_{t_n^+}},
\end{equation}
where
\[
\Xi_{t_n^+}:=\max_{1\leq k\leq t_n^+}\{ \overline{x}_k^\alpha+k\log b\}.
\]
 Recall \eqref{cond}. One only needs to bound $\Xi_{t_n^+}$ under the condition that $\sum_{k=1}^{t_n^+} \overline{x}_k\geq (1-b^{-s_n})(1-2\varepsilon)\ell_n$. By the definition of $\Xi_{t_n^+}$, one sees that
\[
\overline{x}_k\leq \l(\Xi_{t_n^+}-k\log b\r)^{1/\alpha}, \forall k\in\{1,\cdots, t_n^+\}.
\]
So, $\sum_{k=1}^{t_n^+} \overline{x}_k\geq (1-b^{-s_n})(1-2\varepsilon)\ell_n$ yields that
\[
\sum_{k=1}^{t_n^+}  \l(\Xi_{t_n^+}-k\log b\r)^{1/\alpha} \geq (1-b^{-s_n})(1-2\varepsilon)\ell_n.
\]
 Notice that $\Xi_{t_n^+}\geq t_n^+\log b$. By monotonicity of $x\mapsto (\Xi_{t_n^+}-x\log b)^{1/\alpha}$ on $[0,\frac{\Xi_{t_n^+}}{\log b}]$, one has
\[
\sum_{k=1}^{t_n^+}  \l(\Xi_{t_n^+}-k\log b\r)^{1/\alpha}\leq \int_0^{t_n^+} \l(\Xi_{t_n^+}-x\log b\r)^{1/\alpha}dx\leq  \frac{\alpha}{(1+\alpha)\log b} \Xi_{t_n^+}^{1+\frac{1}{\alpha}}.
\]
We then deduce that
\[
\Xi_{t_n^+}\geq \l( \frac{(\alpha+1)\log b}{\alpha} (1-b^{-s_n})(1-2\varepsilon)\ell_n\r)^{ \frac{\alpha}{(\alpha+1)} }.
\]
Going back to \eqref{bdbymax}, one sees that
\begin{multline}
\min\l\{\sum_{u\in\mathbf{t}_b^{\backslash u^*}}e^{x_u^\alpha}\Big\vert x_u\in\mbb N \cap[M,A\ell_n); \forall u\in\mathbf{t}_b^{\backslash u^*}; s_u\geq (1-2\varepsilon)\ell_n; \forall u\in \mathbf{t}_b^{\backslash u^*} s.t. |u|=t_n^+\r\}\\
\geq (1-b^{-s_n}) e^{\Xi_{t_n^+}} \geq (1-b^{-s_n})e^{\l(\frac{\alpha+1}{\alpha}(1-b^{-s_n})(1-2\varepsilon)\log b\r)^{\frac{\alpha}{\alpha+1}} \ell_n^{\frac{\alpha}{\alpha+1}}}.
\end{multline}
Using it to bound $\Sigma_{\mathbf{t}_b^{\backslash u^*}, A}$ tells us that
\[
\Sigma_{\mathbf{t}_b^{\backslash u^*}, A}\leq (CA \ell_n)^{b^{t_n^+}+1} \exp\{-(1-b^{-s_n})e^{\l(\frac{\alpha+1}{\alpha}(1-b^{-s_n})(1-2\varepsilon)\log b\r)^{\frac{\alpha}{\alpha+1}} \ell_n^{\frac{\alpha}{\alpha+1}}}\}.
\]
Plugging it and \eqref{upbd3smallpart} into \eqref{upbd3} implies that
\begin{multline}
\P(M_n\leq m_n-\ell_n)\leq \exp\l(-e^{\eta_0 (\varepsilon\ell_n/2)^{\frac{\alpha}{\alpha+1}}}b^{t_n^+-\delta_n}\r)+o_n(1)\P(M_n\leq m_n-\ell_n)\\
+(CA \ell_n)^{b^{t_n^+}+1} \exp\{-(1+o_n(1))e^{\l(\frac{\alpha+1}{\alpha}(1-2\varepsilon)\log b\r)^{\frac{\alpha}{\alpha+1}} \ell_n^{\frac{\alpha}{\alpha+1}}}\}.
\end{multline}
Here we choose $t^+=[\l(\frac{\alpha+1}{\alpha}(1-2\varepsilon)\log b\r)^{\frac{\alpha}{\alpha+1}}-\eta_\varepsilon/6]/\log b$ and $\delta=\frac{\eta_\varepsilon}{6\log b}$ where $\eta_\varepsilon=\eta_0 (\frac{\varepsilon}{2})^{\frac{\alpha}{\alpha+1}}$ so that
\[
e^{\eta_0 (\varepsilon\ell_n/2)^{\frac{\alpha}{\alpha+1}}}b^{t_n^+-\delta_n}\gg e^{\l(\frac{\alpha+1}{\alpha}(1-2\varepsilon)\log b\r)^{\frac{\alpha}{\alpha+1}} \ell_n^{\frac{\alpha}{\alpha+1}}} \gg b^{t_n^+}\log(CA \ell_n).
\]
This suffices to conclude that
\[
\liminf_{n\rightarrow\infty} \frac{1}{\ell_n^{\frac{\alpha}{\alpha+1}}}\log[-\log\P(M_n\leq m_n-\ell_n)]\geq \l(\frac{\alpha+1}{\alpha}(1-2\varepsilon)\log b\r)^{\frac{\alpha}{\alpha+1}}
\]
for arbitrary small $\varepsilon>0$. This is exactly what we need.

\section{Small ball probability of $D_\infty$ in B\"ottcher case}\label{lefttailD}
This section is devoted  to proving Propositions \ref{weibulltailleft} and \ref{gumbeltailleft}. In fact, we only prove Proposition \ref{weibulltailleft} where $\P(X<-x)= \Theta(1)e^{-\lambda x^\alpha}$. And we feel free to omit the proof of Proposition \ref{gumbeltailleft} as it follows from  similar ideas.

Write $D$ for $D_{\infty}$ for simplicity. It is easy to see that for any time $n\geq1$,
\begin{eqnarray}\label{receq}
D\overset{a.s.}{=}\sum_{|u|=n}e^{\theta^*(S_u-nx^*)}D^{(u)},
\end{eqnarray}
where given $(S_u: |u|=n)$, $\l(D^{(u)}\r)_{\{|u|=n\}}$ are i.i.d. copies of $D$. It is known from \cite{M16} that there exists a constant $C_D>0$ such that as $x\rar+\infty$,
\begin{eqnarray}\label{Dtail}
\P(D>x)\sim \frac{C_D}{x}.
\end{eqnarray}
We only present the proof for \reff{Dweib}. \reff{Dgumb} can be obtained by similar arguments as the proof of Theorem \ref{gumbeltail}.

\subsection{Lower bound}
First observe from \eqref{receq} that for any $n\geq1$ and $\delta>0$,
\begin{align*}
\P(D<\ez)=&\P\l(\sum_{|u|=n}e^{\theta^*(S_u-nx^*)}D^{(u)}<\varepsilon\r)\\
\geq &\P\l(\forall |u|=n, e^{\theta^*(S_u-nx^*)}\leq \ez^{1+\delta}; \sum_{|u|=n}D^{(u)}<\ez^{-\delta}\r)
\end{align*}
where $\sum_{|u|=n}D^{(u)}=\Theta_\P(Z_n \log Z_n)$ because of \eqref{Dtail}. Therefore, by independence,
\begin{align*}
\P(D<\ez)\geq & \P\l(\forall |u|=n, e^{\theta^*(S_u-nx^*)}\leq \ez^{1+\delta}; Z_n=b^{n}; \sum_{|u|=n}D^{(u)}<\ez^{-\delta}\r)\\
=&\P\l(\forall |u|=n, S_u\leq (1+\delta)\frac{\log \ez}{\theta^*}+nx^*; Z_n=b^{n}\r)\P\l(\sum_{k=1}^{b^n}D_k<\ez^{-\delta}\r)
\end{align*}
where $D_k; k\geq1$ are i.i.d. copies of $D$. By weak law for triangular arrays(Theorem 2.2.6 in \cite{Durrett}), $\sum_{k=1}^{b^n}D_k=(C_D+o_\P(1))b^n\log(b^n)$. As long as we take $n=t_\ez\ll \frac{-\delta\log\ez}{\log b}$ so that $nb^{n}\ll \ez^{-\delta}$, $\P\l(\sum_{k=1}^{b^n}D_k<\ez^{-\delta}\r)=1+o(1)$. So for $\ez>0$ small enough,
\begin{align*}
\P(D<\ez)\geq &\frac{1}{2}\P\l(\forall |u|=t_\ez, S_u\leq (1+\delta)\frac{\log \ez}{\theta^*}+t_\ez x^*; Z_{t_\ez}=b^{t_\ez}\r)
\end{align*}
The sequel of this proof will be divided into two parts. Write $a_\ez:=-\log\ez$ for convenience.\\
{\bf Subpart 1: the case $\az>1$}. Choose $t_{\ez}=(\alpha-1)\frac{\log ((1+\dz)a_{\ez})}{\log b}$ and $a_k=\frac{ (b_\alpha-1)(1+2\dz)a_{\ez}}{\theta^*b_\alpha^k}$. Then $\frac{a_{\ez}}{\theta^*}\gg t_{\ez}x^*$ and $\sum_{k=1}^{t_{\ez}}(-a_k)=(1-b_\alpha^{-t_{\ez}})\frac{-(1+2\dz)a_{\ez}}{\theta^*}\leq (1+2\delta)\frac{\log \ez}{\theta^*}+t_\ez x^*$. As a consequence,
\begin{align}\label{Dcon02}
\P(D<\ez)\geq&\frac{1}{2}\P\l(\forall |u|=t_\ez, S_u\leq (1+\delta)\frac{\log \ez}{\theta^*}+t_\ez x^*; Z_{t_\ez}=b^{t_\ez}\r)\nonumber\\
\geq &\frac{1}{2}\P\l(Z_{t_{\ez}}=b^{t_{\ez}}; \forall |u|=k\in \{1,\cdots, t_{\ez}\}, X_u<-a_k \r)\nonumber\\
\geq &\exp\l\{-\lambda \l(\frac{(1+2\dz)a_{\ez}}{\theta^*}\r)^\alpha (b_\alpha-1)^{\alpha-1}- \Theta\l( \l(\frac{(1+\dz)a_{\ez}}{\theta^*}\r)^{\alpha-1}\r)\r\},
\end{align}
where the inequality follows from the same reasonings as \reff{probefore01}.
Letting $\varepsilon\downarrow0$ then $\delta\downarrow0$ implies that
\[
\liminf_{\varepsilon\rar0+}\frac{1}{(-\log\varepsilon)^{\alpha} }\log\P(D_{\infty}<\varepsilon)\geq-\frac{ \lz }{\l(\theta^*\r)^{\alpha} } \l(b^{\frac{1}{\alpha-1}}-1\r)^{\alpha-1}.
\]
{\bf Subpart 2: the case $\az=1$}. Choose $t_{\ez}=1$. Then it follows that
\begin{align}\label{Dcon02a}
\P(D<\ez)\geq &\frac{1}{2}\P\l(\forall |u|=t_\ez, S_u\leq (1+\delta)\frac{\log \ez}{\theta^*}+t_\ez x^*; Z_{t_\ez}=b^{t_\ez}\r)\nonumber\\
=& \P\l(Z_1=b; X_u\leq (1+\delta)\frac{\log \ez}{\theta^*}+x^*, \textrm{ for all }|u|=1 \r)\nonumber\\
\geq & p_b c^be^{\lambda b ((1+\delta)\frac{\log \ez}{\theta^*}+x^*)},
\end{align}
which implies
\[
\liminf_{\varepsilon\rar0+}\frac{1}{(-\log\varepsilon)^{\alpha} }\log\P(D_{\infty}<\varepsilon)\geq-\frac{ \lz (1+\dz)}{\theta^*} b.
\]
Then we obtain the lower bound by letting $\dz\rar0$.
\subsection{Upper bound}
\paragraph{Subpart 1: the case $\alpha>1$.}
Define
\[
U_0(t, \ell):=\{u\in {\cal T}: |u|=t \textrm{ and } {\theta^*(S_u-tx^*)}\geq\ell\}.
\]
We first consider the case $\az>1$. Observe that
\begin{align}\label{Dcon05}
\P\l(D<\ez\r)&=\P\l(\sum_{|u|=t}e^{\theta^*(S_u-tx^*)}D^{(u)}<\ez\r)\nonumber\\
&\leq \P\l( e^{\theta^*(S_u-tx^*)}D^{(u)}<\ez, \, \forall |u|=t\r)
\end{align}
We first obtain a rough bound. In fact,
\begin{align}\label{Dcon05+}
\P\l(D<\ez\r)&\leq \P\l(D^{(u)}<1,\,\forall u\in U_0(t, \log\ez); \#U_0(t,\log\ez)\geq  b^{t}\r)+ \P\l(\#U_0(t,\log\ez)< b^t\r)\nonumber\\
&\leq \P(D<1)^{ b^t}+ \P\l(Z_{t}\l(\l[\frac{\log\ez}{\theta^*}+tx^*,\infty\r)\r)< b^{t}\r)\nonumber\\
&\leq e^{-c b^t}+ \P\l(\sum_{|u|=t}1_{\{S_u\leq \frac{\log\ez}{\theta^*}+tx^*\}}\geq1\r),
\end{align}
because $\P(D<1)<1$ and $Z_t\geq b^t$. Similar to \eqref{upperbd1part2}, by Markov inequality,
\begin{equation*}
 \P\l(\sum_{|u|=t}1_{\{S_u\leq \frac{\log\ez}{\theta^*}+tx^*\}}\geq1\r)\leq e^{-\theta \frac{a_\ez}{\theta^*}+\Theta(t)}
\end{equation*}
for any $\theta>0$ such that $\E[e^{-\theta X}]<\infty$. We take $t=2\log a_\ez/\log b$ so that $b^t\gg a_\ez$ and $t\ll a_\ez$. Then, if $\alpha>1$, for $\ez>0$ small enough,
\begin{equation}\label{roughD}
\P\l(D<\ez\r)\leq e^{-2 a_\ez},
\end{equation}
where $a_\ez=-\log\ez$.

Now again by \eqref{Dcon05}, for any $\delta\in(0,1)$, $t_\ez\in\mathbb{N}_+$ and $\delta_\ez\in (0,t_\ez)\cap\mathbb{N}$,
\begin{align}\label{Dcon06}
&\P(D<\ez)\nonumber\\
\leq & \P\l(\sup_{u\in U_0(t_\ez, (1-\delta)\log\ez)}D^{(u)}< \ez^\delta, \#U_0(t_\ez, (1-\delta)\log\ez)\geq b^{t_\ez-\delta_\ez}\r)+ \P\l(\#U_0(t_\ez,(1-\delta)\log\ez)< b^{t_\ez-\delta_\ez}\r)\nonumber\\
\leq & \P(D<\ez^\delta)^{b^{t_\ez-\delta_\ez}}+\P\l(\#U_0(t_\ez,(1-\delta)\log\ez)< b^{t_\ez-\delta_\ez}\r).
\end{align}
By \eqref{roughD}, one sees that
\[
\P(D<\ez^\delta)^{b^{t_\ez-\delta_\ez}}\leq e^{2\delta b^{t_\ez-\delta_\ez}\log\ez}.
\]
On the other hand, for the second term on the r.h.s. of \reff{Dcon06}, by taking $t_\ez=\Theta(\log a_\ez)\ll a_\ez$,
\begin{align*}
\P\l(\#U_0(t_\ez,(1-\delta)\log\ez)< b^{t_\ez-\delta_\ez}\r)\leq &\P\l(\sum_{|u|=t_\ez}1_{S_u\geq t_\ez x^*+(1-\delta)\frac{\log \ez}{\theta^*}}<b^{t_\ez-\dz_\ez}\r)\\
\leq &\P\l(\sum_{|u|=t_\ez}1_{S_u\geq -(1-2\delta)\frac{a_\ez}{\theta^*}}<b^{t_\ez-\dz_\ez}\r).
\end{align*}
which by the same arguments for deducing \reff{keybd}, is less than
\[
Cb^{t_\ez+1}e^{-\lambda (A a_\ez)^\alpha}+(CA a_\ez)^{b^{t_\ez+1}}\exp\{-\lambda(b_\alpha-1)^{\alpha-1}(\frac{a_\ez}{\theta^*})^{\alpha}(1-4\delta)^\alpha(1+o_\ez(1))\}.
\]
Consequently, \eqref{Dcon06} becomes that
\[
\P(D<\ez)\leq e^{-2\delta b^{t_\ez-\delta_\ez}a_\ez}+Cb^{t_\ez+1}e^{-\lambda (A a_\ez)^\alpha}+(CA a_\ez)^{b^{t_\ez+1}}\exp\{-\lambda(b_\alpha-1)^{\alpha-1}(\frac{a_\ez}{\theta^*})^{\alpha}(1-4\delta)^\alpha(1+o_\ez(1))\}.
\]
Let $t_\ez=\frac{\alpha-1/3}{\log b}\log a_\ez$, $\delta_\ez=\frac{1/3}{\log b}\log a_\ez$ and $A\geq1$ be a large constant so that
\[
 b^{t_\ez-\delta_\ez}a_\ez\gg a_\ez^\alpha\gg b^{t_\ez}\log(CAa_\ez),\quad A^\alpha \geq \frac{2}{\theta^*}(b_\alpha-1)^{\alpha-1}.
\]
This implies that for any $\delta\in(0,1/4)$,
\[\limsup_{\ez\downarrow0}\frac{1}{(-\log \ez)^{\az}}\log\P\l(D<\ez\r)\leq -\frac{\lambda}{(\theta^*)^\alpha}(b_\alpha-1)^{\alpha-1}(1-4\delta)^\alpha,
\]
which gives the upper bound for the case $\az>1$.

\paragraph{Subpart 2: the case $\az=1$.}
For $\dz\in(0,1/b)$, similar to \reff{Dcon05+}, we have, for any $t_{\ez}\in (0, a_{\ez})\cap\mathbb{N}$,
\begin{align}\label{Dupbd2}
\P(D<\ez)\leq & \P\l(D^{(u)}<1,\,\forall u\in U_0(t_\ez, \log\ez); \#U_0(t_\ez,\log\ez)\geq  \delta b^{t_\ez}\r)+ \P\l(\#U_0(t_\ez,\log\ez)< \delta b^{t_\ez}\r)\nonumber\\
\leq & \P(D<1)^{\delta b^{t_\ez}}+\P\l(\#U_0(t_\ez,\log\ez)< \delta b^{t_\ez}\r)\nonumber\\
\leq & e^{-c\delta b^{t_\ez}}+ \P\l(\sum_{|u|=t_\ez}1_{S_u\geq t_\ez x^*-\frac{a_\ez}{\theta^*}}<\delta b^{t_\ez}\r).
\end{align}
Note that for $t_\ez=\Theta(\log a_\ez)\leq \delta' a_\ez$ with some $\delta'\in (0,1)$,
\begin{align*}
\P\l(\sum_{|u|=t_\ez}1_{S_u\geq t_\ez x^*-\frac{a_\ez}{\theta^*}}<\delta b^{t_\ez}\r)=&\P\l(Z_{t_\ez}[t_\ez x^*-\frac{a_\ez}{\theta^*},\infty)<\delta b^{t_\ez}\r)\\
\leq & \P\l(Z_{t_\ez}[-(1-\delta')\frac{a_\ez}{\theta^*},\infty)<\delta b^{t_\ez}\r),
\end{align*}
which by the same reasonings as \reff{upbd20part2}, yields that
\[
 \P\l(Z_{t_\ez}[-(1-\delta')\frac{a_\ez}{\theta^*},\infty)<\delta b^{t_\ez}\r)\leq e^{-\theta b (1-\delta')\frac{a_\ez}{\theta^*}+\Theta(t_\ez)},
\]
for any $\theta\in(0,\lambda)$. Going back to \eqref{Dupbd2}, one sees that
\[
\P(D<\ez)\leq e^{-c\delta b^{t_\ez}}+e^{-\theta b (1-\delta')\frac{a_\ez}{\theta^*}+\Theta(t_\ez)}.
\]
By taking $t_\ez=\frac{2}{\log b}\log a_\ez$ and $\theta=\lambda(1-\delta')$, one obtains that for any $\delta'\in(0,1)$,
\[
\limsup_{\ez\downarrow0}\frac{1}{-\log\ez}\log\P(D<\ez)\leq - \frac{\lambda b}{\theta^*}(1-\delta')^2.
\]
The the desired upper bound for the case $\az=1$ follows obviously.

\section{Moderate deviation in Schr\"oder case: proof of Theorem \ref{Schroder}}
Recall that $M_n:=\max_{|u|=n}\{S_u\}$. In Schr\"oder case, let $\max\emptyset:=-\infty$ for convenience. Then A\"id\'ekon in \cite{A13} proved that for any $x\in\mbb R$,
\begin{equation}
\lim_{n\rightarrow\infty}\P(M_n\leq m_n+x)=\E[e^{-C e^{-x}D_\infty}],
\end{equation}
where $C>0$ is some constant and $D_\infty$ is the a.s. limit of derivative martingale which is a.s. $0$ on the extinction set $\{\cal T<\infty\}$. Therefore,
\[
\lim_{n\rightarrow\infty}\P^s(M_n\leq m_n+x)=\E^s[e^{-C e^{-x}D_\infty}],
\]
which means that $M_n-m_n$ converges in law to some real-valued random variable under $\P^s$.

The idea to obtain Theorem \ref{Schroder} is borrowed from \cite{GH18}. We first recall some results in the literatures, which will be used later.
The idea to this proof is borrowed from \cite{GH18}. We first recall some results from existed literatures.
The following result is the well-known Cram\'er theorem; see Theorem 3.7.4 in \cite{DZ98}.
\begin{lemma}
Under the assumption \reff{expmom},  we have for any  $a>0$, as $n\rightarrow\infty$,
\begin{eqnarray}\label{LDPRW}
\lim_{n\rightarrow\infty}\frac{1}{n}\log\P(S_n\leq -a n)= -I(-a).
\end{eqnarray}
\end{lemma}
The next two statements characterize asymptotic behaviors of lower deviation probability for Galton-Watson process; see Corollary 5 in \cite{FW07} or Proposition 3 in \cite{FW08}. Define $b_1:=\min\{k\geq1: p_k>0\}$ and recall $\gamma=\log f'(q)$.
\begin{lemma}
Assume (\ref{momoffspring}) and $0<p_0+p_1<1$. Then  for the minimal positive offspring number $b$,
\begin{eqnarray}\label{LDPGW}
\lim_{n\rightarrow\infty}\frac{1}{n}\log\P^s(Z_{n}=b_1)
=\lim_{n\rightarrow\infty}\frac{1}{n}\log\P(Z_{n}=b_1) =\gamma,
\end{eqnarray}
and for every subexponential sequence $a_n$ with $a_n\rar\infty$,
\begin{eqnarray}\label{LDPGWa}
\lim_{n\rightarrow\infty}\frac{1}{n}\log\P^s(Z_{n}\leq a_n)=\gamma.
\end{eqnarray}
\end{lemma}
We also have the following fact whose proof can e.g. be found in Lemma 1.2.15 in \cite{DZ98}. For $i\geq 1$, let $(a^i_n)_{n\geq 1}$ be a sequence of positive numbers and $a^i = \limsup_{n\rar\infty}\frac{1}{n} \log a^i_n.$
Then, for all $k\geq 2$ it holds that
\begin{eqnarray}\label{inequality}
\limsup_{n\rar\infty}\frac{1}{n} \sum_{i=1}^k\log a^i_n=\max_{i\in\{1,\cdots,k\}}a^i.
\end{eqnarray}

\subsection{Lower bound}
For the lower bound, we consider the case that there are only $b$ particles at some generation $t_n$, and the random walk of one of those $b_1$-particles moves to the level $-at_n$. Furthermore, families induced by other $b_1-1$ particles at $t_n$-th generation die out before time $n$. For any $\varepsilon>0$ and $y\geq (x^*-\ell^*)\vee0$ such that $a=\ell^*-x^*+2\varepsilon+y^*>0$,  let $t_n=\lceil\frac{\ell_n}{\ell^*+y^*+\varepsilon}\rceil$. Note that $t_n< n$ for $n$ large enough.  By using Markov property at time $t_n$, we have for $n$ large enough,
\begin{align}
&\P^s(M_n\leq m_n-\ell_n)\cr&\geq \P^s(\max_{|u|=t_n}M_{n-t_n}^u\leq m_n+at_n-\ell_n|Z_{t_n}=b_1)\P(S_{t_n}\leq -a t_n)\P^s(Z_{t_n}=b_1)\cr&\geq \P(Z_{n-t_n}=0|Z_0=b_1-1)\P^s(M_{n-t_n}\leq m_n+at_n-\ell_n)\P(S_{t_n}\leq -a t_n)\P^s(Z_{t_n}=b_1)\cr
&\geq (q/2)^{b_1-1}\P^s(M_{n-t_n}\leq m_n+at_n-\ell_n)\P(S_{t_n}\leq -a t_n)\P^s(Z_{t_n}=b_1),
\end{align}
where in the last inequality we use the fact that $\lim_{n\rar\infty}\P(Z_{n-t_n}=0|Z_0=b_1-1)=q^{b_1-1}.$
\noindent Recall that $m_n=x^*n-\frac{3}{2\theta^*}\log n$. Then one can check for $n$ large enough,
\[
m_n+at_n-\ell_n- m_{n-t_n}=(\ell^*+2\varepsilon+y^*)t_n+\frac{3}{2\theta^*}\log\l(\frac{n-t_n}{n}\r)\geq0.
\]
Thus
\[\liminf_{n\rar\infty}\P^s(M_{n-t_n}\leq m_n+at_n-\ell_n)>0\]
and then for $n$ large enough,
\begin{align}
\P^s(M_n\leq m_n-\ell_n)&\geq C_1\P(S_{t_n}\leq -a t_n)\P^s(Z_{t_n}=b_1).
\end{align}
This, with (\ref{LDPRW}) and (\ref{LDPGW}) yields
\[
\liminf_{n\rar\infty}\frac{1}{\ell_n}\log \P^s(M_n\leq m_n-\ell_n)\geq \frac{-I(-a)-\gamma}{\ell^*+y^*+\varepsilon}.
\]
Letting $\varepsilon\downarrow0$, together with the fact that r.h.s. is independent of $y$, gives
\[
\liminf_{n\rar\infty}\frac{1}{\ell_n}\log \P^s(M_n\leq m_n-\ell_n)\geq \sup_{y\geq (x^*-\ell^*)\vee 0}\frac{-I(x^*-\ell^*-y^*)+\gamma}{\ell^*+y^*}.
\]
\qed
\subsection{Upper bound}

 Let \[T_n=\inf\{t\geq 0: Z_{t\ell_n}\geq \ell_n^3\}\]
and for $\delta>0$ and $\varepsilon>0$ small enough set
\[
F(\dz)=\l\{\dz, 2\dz,\cdots, \lceil \frac{1}{\dz(\ell^*\vee x^*)(1+2\varepsilon)}\rceil\delta\r\}.
\]
Then
\begin{align}\label{Schroder01}
&\P^s(M_n\leq m_n-\ell_n)\cr
&\leq \P^s\l(Z_{\frac{\ell_n}{(\ell^*\vee x^*)(1+2\varepsilon)}}\leq \ell_n^3\r)+ \sum_{t\in F(\dz)}\P^s\l(M_n\leq m_n-\ell_n; T_n\in (t-\dz, t]\r).
\end{align}
Note that by \reff{LDPGWa},
\begin{align}\label{Schroder02}\lim_{n\rar\infty}\frac{1}{\ell_n}\log\P^s\l(Z_{\frac{\ell_n}{(\ell^*\vee x^*)(1+2\varepsilon)}}\leq \ell_n^3\r)=\frac{\gamma}{(\ell^*\vee x^*)(1+2\varepsilon)}\end{align}
and
\begin{align}\label{Schroder03}
\limsup_{n\rar\infty}\frac{1}{\ell_n}\log\P^s\l(T_n\in (t-\dz, t]\r)\leq\lim_{n\rar\infty}\frac{1}{\ell_n}\log\P^s\l(Z_{(t-\dz)\ell_n}\leq \ell_n^3\r)=\gamma(t-\dz).
\end{align}
Meanwhile,
\begin{align*}
&\P^s(M_n\leq m_n-\ell_n| T_n\in (t-\dz, t])\cr
&=\P^s(\max_{|u|=t\ell_n}S_{u}+M_{n-t\ell_n}^u\leq m_n-\ell_n | T_n\in (t-\dz, t])\cr
&\leq\P^s(\max_{|u|=t\ell_n}S_{t\ell_n}+M_{n-t\ell_n}^u\leq m_n-\ell_n | T_n\in (t-\dz, t])\cr
&\leq \P(S_{t\ell_n}\leq m_n-(1-\varepsilon)\ell_n -m_{n-t\ell_n})+ \P^s(\max_{|u|=t\ell_n}M_{n-t\ell_n}^u\leq m_{n-t\ell_n}-\varepsilon\ell_n| T_n\in (t-\dz, t])\cr
&=:I_1+I_2,
\end{align*}
where in the first inequality, we use  Lemma 5.1 \cite{GH18} and the fact that $(S_u)$ and $(M_{n-t\ell_n}^u)$ are independent.
We first estimate $I_1$. For any $t\in F(\delta)$, one can check that  $tx^*-1+\varepsilon<0$ and
\begin{align*}
 m_n-(1-\varepsilon)\ell_n -m_{n-t\ell_n}&=\frac{3}{2\theta^*}\log\l(\frac{n-t\ell_n}{n}\r)+(tx^*-1+\varepsilon)\ell_n\cr
 &\leq (tx^*-1+\varepsilon)\ell_n.
\end{align*}
Thus
\[\limsup_{n\rar\infty}\frac{1}{\ell_n}\log\P(S_{t\ell_n}\leq m_n-(1-\varepsilon)\ell_n -m_{n-t\ell_n})\leq -tI\l(\frac{tx^*-1+\varepsilon}{t}\r).
\]
Next, we turn to $I_2$.
\begin{align*}
I_2
&=\E^s[ \P^s(M_{n-t\ell_n}\leq m_{n-t\ell_n}-\varepsilon\ell_n )^{Z_{t\ell_n}}| T_n\in (t-\dz, t]] \cr
&\leq \P^s(M_{n-t\ell_n}\leq m_{n-t\ell_n}-\varepsilon\ell_n )^{\ell_n^2}+ \P^s(Z_{t\ell_n}\leq \ell_n^2|T_n\in (t-\dz, t]).
\end{align*}
Notice that as $\P^s(T_n\in (t-\dz, t])\geq \frac{1-q^{n^3}}{1-q}\P(T_n\in (t-\dz, t])$,
\begin{align*}
\P^s(Z_{tn}\leq \ell_n^2|T_n\in (t-\dz, t])\leq (1-q+o(1))\P(\exists k\leq \dz n,\, Z_k\leq \ell_n^2 |Z_0=\ell_n^3)\leq C_{\ell_n^3}^{\ell_n^2}q^{\ell_n^3-\ell_n^2}
\end{align*}
and  by Theorem 1.1 in \cite{A13}, we have there exists $c^*\in (0, \infty]$ such that
\[\lim_{n\rar\infty}\P^s(M_{n-t\ell_n}\leq m_{n-t\ell_n} -\varepsilon\ell_n)=e^{-c^*}<1.\]
Thus $I_2\leq e^{-c_1\ell_n^2}$ and hence
\begin{align}\label{Schroder04}
\limsup_{n\rar\infty}\frac{1}{\ell_n}\log\P^s(M_n\leq m_n-\ell_n| T_n\in (t-\dz, t])\leq -tI\l(\frac{tx^*-1+\varepsilon}{t}\r).
 \end{align}
Going back to \reff{Schroder01}, together with \reff{Schroder02}, \reff{Schroder03} and \reff{inequality}, one has
\begin{align*}
&\limsup_{n\rar\infty}\frac{1}{\ell_n}\log\P^s(M_n\leq m_n-\ell_n)\cr&\leq \frac{\gamma}{(\ell^*\vee x^*)(1+2\varepsilon)}\vee\sup_{t\in F(\delta)} \l((t-\delta)\gamma-tI\l(\frac{tx^*-1+\varepsilon}{t}\r)\r),
\end{align*}
which by letting $\varepsilon \downarrow0$ and $\dz \downarrow0$ implies
\begin{align}
&\limsup_{n\rar\infty}\frac{1}{\ell_n}\log\P^s(M_n\leq m_n-\ell_n)\cr&\leq \sup_{t\in \l(0, \frac{1}{\ell^*\vee x^*}\r)} \l(t\gamma-tI\l(\frac{tx^*-1}{t}\r)\r)\cr
&= \sup_{y\geq (x^*-\ell^*)_+} \frac{-I(x^*-\ell^*-y^*)+\gamma}{\ell^*+y^*}.
\end{align}
We have completed the proof. \qed

\bigskip\bigskip

{\bf Acknowledgement.} We are grateful to Elie A\"{i}d\'{e}kon and Yueyun Hu for enlightening discussions. Hui He is supported by NSFC (No. 11671041, 11531001).

\bigskip\bigskip

\noindent{\small Xinxin Chen}

\noindent{Institut Camille Jordan, C.N.R.S. UMR 5208, Universite Claude Bernard Lyon 1, 69622 Villeurbanne Cedex, France.}

\noindent{E-mail: {\tt xchen@math.univ-lyon1.fr}}

\bigskip

\noindent{\small Hui He}

\noindent{School of Mathematical Sciences, Beijing Normal University,
Beijing 100875, People's Republic of China.}

\noindent{E-mail: {\tt hehui@bnu.edu.cn}}

\end{document}